\documentclass{article}

\usepackage{amsmath}
\usepackage[T1]{fontenc}
\usepackage[utf8]{inputenc}
\usepackage[english]{babel}
\usepackage{csquotes}
\usepackage{amsfonts}
\usepackage{amsthm}
\usepackage{amssymb}
\usepackage{enumitem}
\usepackage{xcolor}
\usepackage{mathabx}
\usepackage{graphicx}
\usepackage{changepage}
\usepackage{bbold}
\usepackage{tikz}
\usepackage{tikz-cd}
\usepackage{tabularx}
\usepackage{mathtools}
\usepackage{bm}
\usepackage{hyperref}
\usepackage{ytableau}
\usepackage{easyReview}
\usepackage[firstinits=true, style=alphabetic]{biblatex}
\usetikzlibrary{arrows}

\usepackage{pgfplots}
\pgfplotsset{width=10cm,compat=1.9}
\usepgfplotslibrary{external}
\newtheorem{theorem}{Theorem}[section]
\newtheorem*{theorem**}{Theorem\theoremnum}
\newenvironment{theorem*}[1][]{%
	\edef\theoremnum{\if\relax\detokenize{#1}\relax\else~#1\fi}
	\begin{theorem**}
	}{%
	\end{theorem**}
}  
\newtheorem{pro}{Proposition}[section]
\newtheorem{cor}{Corollary}[section]
\newtheorem{lem}{Lemma}[section]

\newtheorem*{cor**}{Corollary\theoremnum}
\newenvironment{cor*}[1][]{%
	\edef\theoremnum{\if\relax\detokenize{#1}\relax\else~#1\fi}
	\begin{cor**}
	}{%
	\end{cor**}
} 

\theoremstyle{definition}
\newtheorem{definition}{Definition}[section]
\newtheorem{ex}{Example}[section]

\theoremstyle{remark}
\newtheorem{rem}{Remark}

\hypersetup{
	citecolor = blue,
	colorlinks=true,
	linkcolor=.,
	filecolor=magenta,      
	urlcolor=blue,
	breaklinks=true
}

\AtBeginBibliography{\sloppy}

\graphicspath{ {./images/} }

\author{Riccardo Ontani} 

\AtEndDocument{\bigskip{
  \noindent
  \href{mailto:r.ontani@imperial.ac.uk}{\texttt{r.ontani@imperial.ac.uk}}\\$\!$\\
  \textsc{Department of Mathematics\\Imperial College\\London SW7 2AZ\\United Kingdom}
}}

\title{A Vafa-Intriligator formula for semi-positive quotients of linear spaces.}

\begin{document}
	\maketitle
	\begin{abstract}
		We consider genus zero quasimap invariants of smooth projective targets of the form $V/\!/G$, where $V$ is a representation of a reductive group $G$. In particular we consider integrals of cohomology classes arising as characteristic classes of the universal quasimap. In this setting, we provide a way to express the invariants of $V/\!/G$ in terms of invariants of $V/\!/T$, where $T$ is a maximal subtorus of $G$. Using this, we obtain residue formulae for such invariants as conjectured by Kim, Oh, Yoshida and Ueda. Finally, under some positivity assumptions on $V/\!/G$, we prove a Vafa-Intriligator formula for the generating series of such invariants, expressing them as finite sums of explicit contributions.
	\end{abstract}
	\subsection*{Notation}
	Consider a finite dimensional representation $G \curvearrowright V$ of a reductive algebraic group over $\mathbb{C}$. A character $\xi \in \chi(G)$ of $G$, which we will call \textit{stability parameter}, defines a linearisation $\mathcal{L}\rightarrow V$ and by GIT we obtain a geometric quotient $V^{\text{ss}}\rightarrow V/\!/G$ of the semistable locus. We will always work in the case where the semistable locus has trivial stabilisers. We also assume that $V/\!/T$ (and hence $V/\!/G$) is proper, where $T$ is a maximal subtorus of $G$. Here we specify some notation:
	\begin{itemize}
		\item We will denote with $\chi(G)^\vee$ the dual space of the character lattice $\chi(G)$. If $G = T$ then this coincides with the cocharacters lattice. Given two dual characters $\delta \in \chi(G)^\vee$ and $\tilde{\delta} \in \chi(T)^\vee$, we write $\tilde{\delta} \mapsto \delta$ whenever $\tilde{\delta}_{\vert \chi(G)} = \delta$.
		\item Let $\mathfrak{A}$ be the set of \textit{weights} for the action of $T$ on $V$. This is the set of characters $\rho\in \chi(T)$ for which there exists a nonzero $v \in V$ satisfying $t \cdot v = \rho(t)v$ for all $t \in T$. Let $V_\rho$ be the linear subspace of $V$ spanned by all such vectors $v$. The weight $\rho$ appears repeated $\text{dim}(V_\rho)$ times in $\mathfrak{A}$.
		\item Let $\Delta \subset \chi(T)$ be the set of \textit{roots} of $G$, i.e. the set of weights of the adjoint representation of this group. It's well known that the roots of $G$ come in opposite pairs. For each pair we pick one of the two roots and we call it a \textit{positive root}. Once a choice of positive roots is fixed, we will denote with $\Delta^+ \subset \Delta$ the set of positive roots.
	\end{itemize} 
	\section{Introduction}
	In this paper we consider genus zero $0^+$-stable graph quasimap invariants of targets of the form $V/\!/G$, with no marked points. These invariants arise from moduli spaces that serve as compactifications, introduced in \cite[Section 7.2]{KimCiocanMaulik}, of the spaces parameterising regular maps from $\mathbb{P}^1$ to $V/\!/G$ (or to more general GIT quotients). These spaces are different from the stable map compactifications used in Gromov–Witten theory, although the two are closely related by wall-crossing phenomena \cite{KimCiocanWallCrossing, zhou2022quasimap}. 
	
	These compactifications have provided an effective setting for performing a wide range of enumerative computations. For example, Szenes and Vergne \cite{SzenesVergne} used these spaces to prove the toric residue mirror symmetry conjecture of Batyrev and Materov \cite{BatyrevMaterov}, while Marian and Oprea \cite{MarianOpreaBundles} used known formulae for positive genus quasimap invariants of Grassmannians to recover the intersection theory of the moduli spaces of stable bundles on Riemann surfaces.
	
	In this paper we focus on three different aspects of the theory:
	\begin{enumerate}
		\item We prove an abelianisation formula to relate some quasimap invariants of $V/\!/G$ to invariants of $V/\!/T$.
		\item We prove a residue formula for quasimap invariants, well known for toric varieties, for targets of the form $V/\!/G$ where $G$ is not necessarily abelian.
		\item We generalise a closed formula for a generating series of quasimap invariants, known as the Vafa-Intriligator formula in the case of toric varieties and Grassmannians, to targets of the form $V/\!/G$ satisfying a positivity condition.
	\end{enumerate}
	\subsubsection{An abelianisation formula}
	Quasimap invariants of toric varieties have been widely studied in the past \cite{GiventalMirror, KimModuliToric, SzenesVergne}. In order to understand the more general case of quotients by a reductive group, a popular strategy consists of trying to relate the invariants of $V/\!/G$ to those of $V/\!/T$. This paper uses crucial results in this direction, mainly from Webb's research \cite{WebbAbelianNonAbelian} on the $I$-function, to derive some formulae for quasimap invariants of $V/\!/G$ by reducing the computation to the toric case. Our first result is an \textit{abelianisation formula for quasimap invariants} which relates single (i.e. for fixed degree) invariants of $V/\!/G$ to the invariants of $V/\!/T$:
	\begin{theorem*}[\ref{MainTheoremAbFor}]
		Consider a degree $\delta \in \chi(G)^\vee$ and a class $P \in A_G^\ast(\text{pt})$. The $\mathbb{C}^\ast$-equivariant quasimap invariant of $V/\!/G$ of degree $\delta$ with insertion $P$ can be computed as a sum of $\mathbb{C}^\ast$-equivariant quasimap invariants on $V/\!/T$ of degree $\tilde{\delta}\in \chi(T)^\vee$, where $\tilde{\delta}_{\vert \chi(G)} = \delta$:
		\begin{align*}
			\int_{[Q(V/\!/G, \delta)]^{\text{vir}}} \text{CW}^\delta(P) = \frac{1}{\vert W \vert} \sum_{\tilde{\delta}\mapsto \delta} \int_{[Q(V/\!/T, \tilde{\delta})]^{\text{vir}}} \text{CW}^{\tilde{\delta}}\left(P \prod_{\alpha \in \Delta} D(\tilde{\delta}, \alpha) \right)
		\end{align*} 
		where $W$ is the Weyl group of $G$. Here
		\begin{align*}
			\prod_{\alpha \in \Delta} D(\tilde{\delta}, \alpha) = \prod_{\alpha \in \Delta^+}(-1)^{\langle\tilde{\delta},\alpha\rangle+1} \alpha \cdot (\alpha + \langle \tilde{\delta},\alpha\rangle z)
		\end{align*}
		in $A^\ast_G(\text{pt})[z]$.
	\end{theorem*}
	Here $\text{CW}^\delta$ denotes the following homomorphism. Consider the universal quasimap of degree $\delta$ to $V/\!/G$
	\[\begin{tikzcd}
		\mathcal{P}_\delta \arrow[d, "\pi"] \arrow[r,"u"] & V\\
		\text{Q}(V/\!/G, \delta)\times\mathbb{P}^1 & 
	\end{tikzcd}\]
	where $\pi$ is a (left) $G$-principal bundle and $u$ is a $G$-equivariant morphism, mapping the generic fiber of $\pi$ to the semistable locus. Then $\text{CW}^\delta$ is the composition
	\begin{align*}
		A_G^\ast(\text{pt}) \rightarrow A_G^\ast(\mathcal{P})\xrightarrow{(\pi^\ast)^{-1}} A^\ast(\text{Q}(V/\!/G, \delta)\times \mathbb{P}^1) \xrightarrow{i^\ast} A^\ast(\text{Q}(V/\!/G, \delta)\times \infty).
	\end{align*}
	This morphism equals the usual Chern-Weil homomorphism of the (right) $G$-principal bundle $(\mathcal{P}_\delta)_{\vert \infty}$, defined through Chern-Weil theory (see Remark \ref{ChernWeilRemark}). The right G-action on $(\mathcal{P}_\delta)_{\vert \infty}$ is defined in terms of the left action via $p \cdot_{\text{right}} g := g^{-1}\cdot_{\text{left}} p$.
	\subsubsection{Jeffrey-Kirwan formulae}
	For toric targets, it is well known that the quasimap spaces can be realized as GIT quotients of linear spaces by torus actions (see Example~\ref{exampleToricqmaps1}). Building on this, Szenes and Vergne \cite{SzenesVergne} computed the toric quasimap invariants by applying Jeffrey–Kirwan localisation directly to these moduli spaces, thereby expressing the invariants as Jeffrey–Kirwan residues of explicit rational functions.
	By combining their results with the abelianisation formula from the previous section, we obtain a residue formula for the fixed degree $\mathbb{C}^\ast$-equivariant quasimap invariants of $V/\!/G$. This is described in the following theorem.
	\begin{theorem*}[\ref{MainTheoremSingleInvariants}]
		Let $\delta \in \chi(G)^\vee$ and $P \in A^\ast_G(\text{pt})$. Consider the linear space $\chi(T \times \mathbb{C}^\ast)_\mathbb{C}^\vee \simeq \chi(T)_\mathbb{C}^\vee \times \mathbb{C}$ and let $z$ be the coordinate on the second factor $\mathbb{C}$. Given $\tilde{\delta} \in \chi(T)^\vee$, consider the rational function
		\begin{align*}
			Z_{\tilde{\delta}} = \prod_{\rho \in \mathfrak{A}}D(\tilde{\delta}, \rho)^{-1}\prod_{\alpha \in \Delta}D(\tilde{\delta}, \alpha)
		\end{align*}
		on this linear space, where $D(-,-)$ is defined as
		\begin{align}\label{Dfunction}
			D(\tilde{\delta}, w):= \frac{\prod_{k=-\infty}^{\langle \tilde{\delta}, w \rangle }(w+kz)}{\prod_{k=-\infty}^{-1}(w+kz)} = \begin{cases}
				\prod_{k=\langle \tilde{\delta}, w \rangle+1}^{-1} (w+kz)^{-1} & \text{if } \langle \tilde{\delta}, w \rangle < -1\\
				1 & \text{if } \langle \tilde{\delta}, w \rangle=-1\\
				\prod_{k=0}^{\langle \tilde{\delta}, w \rangle} (w+kz) & \text{if } \langle \tilde{\delta}, w \rangle > -1
			\end{cases}
		\end{align}
		for all $w \in \chi(T)$.
		The $\mathbb{C}^\ast$-equivariant quasimap invariant of degree $\delta$
		\begin{align*}
			\int_{[Q(V/\!/G, \delta)]^{\text{vir}}} \text{CW}^\delta(P) \in \mathbb{C}[z]_z
		\end{align*} 
		evaluated at a generic $z \in \mathbb{C}$ coincides with the sum of Jeffrey-Kirwan residues
		\begin{align*}
			\frac{1}{\vert W \vert}\sum_{\tilde{\delta}\mapsto \delta} \sum_{\tilde{\delta}_1+ \tilde{\delta}_2= \tilde{\delta}} \text{JK}^{\mathfrak{A}}_{-z\tilde{\delta}_2} \left(Z_{\tilde{\delta}}(-, z) P \right).
		\end{align*}
		Here $Z_{\tilde{\delta}}(-,z)$ is the rational function on $\chi(T)_\mathbb{C}^\vee$ obtained by evaluating $Z_{\tilde{\delta}}$ at the chosen $z \in \mathbb{C}$.
	\end{theorem*}
	This settles a conjecture by Kim, Oh, Ueda and Yoshida \cite[Conjecture 10.10]{kimRMS} at least in the case of targets of the form $V/\!/G$. By carefully setting the equivariant parameter at zero we recover a residue formula for non-equivariant quasimap invariants:
	\begin{cor*}[\ref{nonequivariantInvariants}]
		Let $\delta \in \chi(G)^\vee$ and $P \in A^\ast_G(\text{pt})$. The non-equivariant quasimap invariant of degree $\delta$
		\begin{align*}
			\int_{[Q(V/\!/G, \delta)]^{\text{vir}}} \text{CW}^\delta(P)
		\end{align*} 
		is the sum of Jeffrey-Kirwan residues
		\begin{align*}
			\frac{1}{\vert W \vert}\sum_{\tilde{\delta}\mapsto \delta} \text{JK}^\mathfrak{A}_{O} \left( \tilde{Z}_{\tilde{\delta}} P \right)
		\end{align*}
		where $\tilde{Z}_{\tilde{\delta}}$ is the rational function on $\chi(T)_\mathbb{C}^\vee$ defined by
		\begin{align*}
			\tilde{Z}_{\tilde{\delta}} = \prod_{\rho \in \mathfrak{A}}\rho^{-1-\langle\tilde{\delta},\rho\rangle} \prod_{\alpha \in \Delta^+} (-1)^{1+\langle \tilde{\delta}, \alpha\rangle} \alpha^2.
		\end{align*}
	\end{cor*}
	\subsubsection{The Vafa-Intriligator formula}
	Up to this point, we have studied quasimap invariants of $V/\!/G$ for fixed degree. Now we consider the behaviour of quasimap invariants of $V/\!/G$ once they are arranged into generating series by summing over the degree. 
	In this part of the work we will assume that that the triple $(T,V,\xi)$ is \textit{semi-positive}, as defined in \cite[Section 1.4]{KimCiocanWallCrossing}. In this section we will mainly be interested in the formal sum
	\begin{align*}
		\langle P \rangle^G(q) := \vert W \vert \sum_{\delta \in \chi(G)^\vee} q^\delta \int_{[Q(V/\!/G, \delta)]^{\text{vir}}} \text{CW}^{\delta}(P).
	\end{align*}
	Formulae for the evaluation of this power series go under the name of\textit{ Vafa-Intriligator formulae}. In case the target $V/\!/G$ is a Grassmannian, these formulae were proven by Bertram \cite{bertram1994towards}, Siebert-Tian \cite{siebert1997quantum} and Marian-Oprea \cite{MarianOpreaVI}, who even established results in higher genus. Similar formulae for comuniuscule homogeneous spaces were found by Chaput-Manivel-Perrin in \cite{chaput2010quantum}. Here we'll find explicit results for semi-positive targets of the form $V/\!/G$.
	\begin{definition}
		Given a group $G$, we will denote with $\widecheck{G}$ the algebraic torus $\chi(G)_\mathbb{C}/\chi(G)$, which we call \textit{dual group of} $G$.
	\end{definition}
	Given $q=[\psi] \in \widecheck{G}$, setting $q^\delta := e^{2\pi i\langle \delta, \psi\rangle}$ we can think of the sum (\ref{QuasimapSeries}) as a function $\langle P \rangle^G : \widecheck{G} \dashrightarrow \mathbb{C}$ defined on its domain of convergence.
	Here we address the following questions:
	\begin{enumerate}
		\item\label{question1} Does $\langle P \rangle^G$ converge at any point of $\widecheck{G}$?
		\item\label{question2} If yes, can we find a closed expression for the function it converges to?
	\end{enumerate}
	Using the abelianisation formula and the analogous result for toric quotients due to Szenes and Vergne \cite{SzenesVergne}, we will answer positively to question \ref{question1} in Theorem \ref{theoremQuestion1}. Moreover, we will find a closed expression for the resulting function in Theorem \ref{theoremQuestion2}. This will be given in terms of the function
	\begin{align*}
		p : \chi(T)_\mathbb{C}^\vee \dashrightarrow \widecheck{T} \quad : \quad p(u):= \sum_{\rho \in \mathfrak{A}} \frac{\log\left(\rho(u)\right)}{2 \pi i}[\rho].
	\end{align*}
	whose fiber is finite for a generic $q\in \widecheck{T}$. 
	\begin{definition}
		Given $q \in \widecheck{T}$ we consider the sum
		\begin{align*}
			\sum_{w \in p^{-1}(q)} \frac{P(w) \prod_{\alpha \in \Delta}\alpha(w)}{D_{\mathfrak{A}}(w)\prod_{\rho \in \mathfrak{A}}\rho(w)},
		\end{align*}
		which is a well defined finite sum for a generic choice of $q$.
	\end{definition}
	\begin{theorem*}[\ref{theoremQuestion2}][Vafa-Intriligator formula]
		Consider a finite dimensional representation $V$ of a reductive group $G$ together with a GIT stability $\xi \in \chi(G)$ so that the action on the semistable locus is free and $V/\!/T$ is proper and semi-positive, where $T \subseteq G$ is a maximal subtorus. Given a class $P \in A^\ast_G(\text{pt})$, the sum above extends to a rational function
		\begin{align*}
			\langle P\rangle_{\mathfrak{B}} : \widecheck{T} \dashrightarrow \mathbb{C}.
		\end{align*}
		Let $q \in \widecheck{G}$ belong to the domain of convergence of the power series $\langle P \rangle^G$. Then it belongs to the domain of definition of $\langle P \rangle_{\mathfrak{B}}$ and 
		\begin{align*}
			\langle P \rangle^G(q) = \langle P \rangle_\mathfrak{B}(\sigma(q)).
		\end{align*}
	\end{theorem*}
	Here $\sigma$ is an involution (\ref{sigmaInvolution}) of $\widecheck{T}$ and $D_\mathfrak{A}$ is a rational function on $\chi(T)^\vee_\mathbb{C}$ obtained as Jacobian determinant of $p$, once expressed in some local coordinates (see Lemma \ref{JacobianP}). The self contained subsection \ref{subsectionVafaIntriligatorFormula} explains this formula in detail.
	\subsection{Outline of the proofs}
	Here we outline the arguments we'll use to prove the results introduced above.
	\subsubsection{Proof of the abelianisation formula}
	First of all, we use the $\mathbb{C}^\ast$-action on the moduli space $Q(V/\!/G,\delta)$ to localise the integrals to the fixed loci. These fixed loci are made of simpler building blocks $F_{\delta, 0}$ and $F_{0, \delta}$, which are substacks of $Q(V/\!/G,\delta)$ endowed with an evaluation morphism to $V/\!/G$. The fixed locus is a disjoint union of fiber products of the form
	\begin{align*}
		Q(V/\!/G,\delta)^{\mathbb{C}^\ast} = \coprod_{\delta_1+\delta_2 = \delta} F_{\delta_1, \delta_2}\quad , \quad F_{\delta_1, \delta_2} := F_{\delta_1, 0} \times_{V/\!/G} F_{0,\delta_2}
	\end{align*}
	and we are left with computing the contribution of such loci to the integral. The localised virtual fundamental classes of these fixed loci preserve the fiber product structure (see \cite[Section 4]{KimCiocanWallCrossing}, recalled in Proposition \ref{splitting_virtual_classes}) and the same is true for the restriction of the Chern classes of the universal principal bundle (Proposition \ref{splittingOfCW}), therefore we manage to split the pushforward of the localised class from $F_{\delta_1, \delta_2}$ to $V/\!/G$ as a product of pushforwards of classes originally defined on $F_{\delta_1, 0}$ and $F_{0, \delta_2}$ respectively:
	\begin{align*}
		&\text{ev}_\ast \left(\text{contribution from }F_{\delta_1, \delta_2}\right)\\
		= &\text{ev}_\ast\left(\text{contribution from } F_{\delta_1, 0}\right) \cdot \text{ev}_\ast\left(\text{contribution from } F_{0,\delta_2}\right).
	\end{align*}
	Then we adapt an argument of Webb \cite{WebbAbelianNonAbelian} to compute the two factors coming from $F_{\delta_1, 0}$ and $F_{0,\delta_2}$ and to relate them to quasimap invariants of $V/\!/T$, thus recovering the abelianisation formula of Theorem \ref{MainTheoremAbFor}.
	\subsubsection{Proof of the Jeffrey-Kirwan formulae}
	Now we have reduced nonabelian quasimap invariants to abelian ones. Since quasimap spaces for abelian targets are toric varieties, we can use the Jeffrey-Kirwan localisation formula \cite{JeffreyKirwan} in the form of Brion, Szenes and Vergne \cite[Proposition 2.3]{SzenesVergne} to express the integrals as JK residues \cite{BrionVergne}. As a result we obtain the JK formula of Theorem \ref{MainTheoremSingleInvariants} for $\mathbb{C}^\ast$-equivariant quasimap invariants.
	
	Here we remark that we could have obtained the same formula directly, without passing to the quasimap spaces for abelian targets. By following the same argument outlined for the proof of the abelianisation formula, we could have pushed the computation onto the target $V/\!/G$. Then, nonabelian JK localisation (for example in the form of \cite{OntaniJKforDT}) would have expressed the integral as a JK residue recovering the same formula above.
	
	The reason why we find it convenient to pass to the moduli spaces for the abelian targets is that this makes it easier to prove the non-equivariant version of this result, contained in Corollary \ref{nonequivariantInvariants}. Indeed, the abelianisation formula also holds for nonequivariant invariants by simply setting the equivariant parameter $z$ to zero, and then we can compute the nonequivariant abelian quasimap invariants by JK localisation. 
	\subsubsection{Proof of the Vafa-Intriligator formula}
	In \cite[Theorem 4.1]{SzenesVergne}, by using the nonequivariant JK formulae described above, Szenes and Vergne prove the Vafa-Intriligator formula for toric targets. In the appendix \ref{sectionSzenesVergne} we recall their result and, in Theorem \ref{SzenesVergneVers1}, we express it into a slightly different form better fit for our applications. For a general target of the form $V/\!/G$, by using the abelianisation formula we can express the generating series $\langle P \rangle^G$ as the restriction (twisted by an involution) to $\widecheck{G}$ of the corresponding abelian generating series $\langle P \rangle^T$ defined in (\ref{QuasimapSeriesAb}) on the torus $\widecheck{T}$ (Theorem \ref{AbelianisationSeries}). By \cite{SzenesVergne} we know that $\langle P \rangle^T$ converges somewhere and we prove that this domain of convergence intersects $\widecheck{G}$, thus showing that $\langle P \rangle^G$ converges on a analytic open subset of $\widecheck{G}$ (Theorem \ref{theoremQuestion1}). Moreover, the work of Szenes and Vergne gives an explicit formula for $\langle P \rangle^T$ on its domain of convergence, therefore we obtain the Vafa-Intriligator formula for $\langle P \rangle^G$ in the form of Theorem \ref{theoremQuestion2}.
	\section*{Acknowledgements}
	I wish to thank Richard Thomas, Jeongseok Oh, Shubham Sinha, Rachel Webb and Weihong Xu for many discussions related to this work. Finally, I would like to express my gratitude to Denis Nesterov and Dragos Oprea for their comments on the paper.
	\tableofcontents
	\section{Quasimap invariants}\label{SectionQuasimapInvariants}
	Following the notation of \cite[Section 2.1]{WebbAbelianNonAbelian}, for us a $G$-principal bundle on a scheme $X$ will be a scheme $P$ with a free left $G$-action and a $G$-invariant map to $X$ that is locally trivial in the \'etale topology. All principal bundles have a left action, unless we directly specify that we are considering a right action.
	\subsection{The Chern-Weil morphism}
	Given $\delta \in \chi(G)^\vee$ we will denote with $Q(V/\!/G,\delta)$ the moduli stack of $0^+$-stable genus zero unmarked graph quasimaps to $V/\!/G$ as described, for example, in \cite[Section 2.6]{KimCiocanWallCrossing}.
	Here we recall a way to define cohomology classes on quasimap spaces. First of all, notice that since a graph quasimap is a couple of a principal $G$-bundle $P\rightarrow \mathbb{P}^1$ and a $G$-equivariant map $P\rightarrow V$, there is a universal quasimap
	\[\begin{tikzcd}
		\mathcal{P}_\delta \arrow[d, "\pi"] \arrow[r,"u"] & V\\
		\text{Q}(V/\!/G, \delta)\times\mathbb{P}^1 & 
	\end{tikzcd}\]
	where $\pi$ is a principal $G$-bundle and $u$ is $G$-equivariant. Chosen a point $p \in \mathbb{P}^1$ we can restrict the diagram above to $\text{Q}(V/\!/G, \delta)\times\lbrace p \rbrace$:
	\begin{equation}\label{CWdiagram}
		\begin{tikzcd}
			\mathcal{P}_{\delta\vert p} \arrow[d, "\pi"] \arrow[r,"u_p"] & V\\
			\text{Q}(V/\!/G, \delta)& 
		\end{tikzcd}
	\end{equation}
	where $\mathcal{P}_{\delta\vert p}$ is the restriction of $\mathcal{P}_\delta$. In particular, we can consider the \textit{Chern-Weil homomorphism} defined as
	\begin{align}\label{CWnoneq}
		CW^\delta : A^\ast_G(\text{pt}) \rightarrow A^\ast_G(\mathcal{P}_{\delta\vert p}) \xrightarrow{d} A^\ast(\text{Q}(V/\!/G, \delta))
	\end{align}
	where $d$ is the isomorphism of \cite[Theorem 4]{EdidinEIT} for the free action of $G$ on $\mathcal{P}_{\delta \vert p}$ (this is the isomorphism saying that a $G$-equivariant Chow class on the total space of the $G$-principal bundle $\mathcal{P}_{\delta \vert p}$ corresponds to a class on the base $Q(V/\!/G, \delta))$.
	\begin{rem}\label{ChernWeilRemark}
		Notice that $\text{CW}^\delta$ is the usual Chern-Weil homomorphism given by pulling back classes through the classifying map to $BG$ induced by the \textbf{right} $G$-principal bundle $\mathcal{P}_{\delta \vert p}$. This is easy to see in the case where $G$ is a torus, while the general case follows from the splitting principle. Indeed, consider a left principal $G$-bundle $P \rightarrow X$ and the composition
		\begin{align*}
			\text{CW} : A^\ast_G(\text{pt}) \rightarrow A^\ast_G(P) \simeq A^\ast(X).
		\end{align*}
		In the case of $G:=\mathbb{C}^\ast$ we find by construction that the classical Chern-Weil homomorphism maps $t$ into the first Chern class of the vector bundle associated to the right principal bundle $\mathcal{P}$, whose total space is the quotient of $P \times \mathbb{C}$ by the action
		\begin{align*}
			\mathbb{C}^\ast \curvearrowright P \times \mathbb{C} \quad : \quad t\cdot (p, z) := (p \cdot t^{-1}, t z) = (t \cdot p, tz).
		\end{align*}
		In particular, if $G= \text{GL}_r$, and $P$ is the frame bundle of a vector bundle $E$, endowed with the usual left $G$-action, then $\text{CW}$ is the Chern-Weil morphism of the dual bundle $E^\vee$. In other words, the morphism $\text{CW}$ maps a symmetric polynomial $P(x_1, \dots, x_r)$ to the cohomology class $P(\alpha_1, \dots, \alpha_r)$ of $X$ given by evaluating $x_i$ at the Chern roots $\alpha_i$ of the dual vector bundle $E^\vee$.
	\end{rem}
	\begin{rem}
		Notice that the morphism $\text{CW}^\delta$ doesn't depend on the choice of $p\in \mathbb{P}^1$ since different points yield homotopic maps, hence the same pullback in cohomology.
	\end{rem}
	In order to produce invariants from these cohomology classes we need a Chow homology class to pair with. Luckily, these moduli spaces carry a perfect obstruction theory and a virtual fundamental class $[Q(V/\!/G, \delta)]^{\text{vir}}$ as discussed in \cite[Theorem 7.2.2]{KimCiocanMaulik}. This class is built in terms of the universal quasimap so that the virtual tangent bundle is $R{\pi_1}_\ast u^\ast T_{V/G}$ where $u$ is considered as a map from $\text{Q}(V/\!/G, \delta)\times\mathbb{P}^1$ to the quotient stack $V/G$.
	Now we can give the following definition:
	\begin{definition}
		Let $P \in A_G^\ast(\text{pt})$. The \textit{graph-quasimap invariant} of $V/\!/G$ of class $\delta$ with insertion $P$ is the number
		\begin{align*}
			\int_{[Q(V/\!/G, \delta)]^{\text{vir}}} \text{CW}^\delta(P) \in \mathbb{Q}.
		\end{align*}
	\end{definition}
	\begin{ex}\label{exampleToricqmaps1}
		Here we recall the explicit form of the quasimap spaces in the case where $G=T$ is a torus. This construction goes back to Morrison-Plesser \cite{MorrisonPlesser} and Givental \cite[Section 5]{GiventalMirror}. We can diagonalise the action and consider the case of $(\mathbb{C}^\ast)^n \curvearrowright \mathbb{C}^m$. We have the isomorphism $\text{Pic}^T(V)\simeq \chi(T)^\vee\simeq \mathbb{Z}^n$ and a quasimap of degree $a = (a_1, ..., a_n)$ is (isomorphic to) a diagram of the form
		\[\begin{tikzcd}
			\mathcal{O}(-a_1) \oplus \dots \oplus \mathcal{O}(-a_n) \arrow[r, "\varphi"]\arrow[d, "\pi"]& \mathbb{C}^m\\
			\mathbb{P}^1 &
		\end{tikzcd}\]
		In other words, a $0^+$-stable genus zero graph quasimaps of degree $a$ identifies a section of the bundle $\mathcal{O}(-a_1) \oplus \dots \oplus \mathcal{O}(-a_n) \times_T \mathbb{C}^m$, hence an element of the linear space of sections
		\begin{align}\label{quasimapSections}
			L(a):= \oplus_{j=1}^m H^0\left(\mathbb{P}^1, \mathcal{O}(\langle a, \rho_j\rangle)\right)
		\end{align}
		where $\rho_1, \dots, \rho_m \in \chi(T)$ are the weights for the $T$-action on $\mathbb{C}^m$. Not all such sections are stable quasimaps: the stability condition identifies an open subscheme 
		\begin{align}\label{quasimapSectionsSS}
			\left(\oplus_{j=1}^m H^0\left(\mathbb{P}^1, \mathcal{O}(\langle a, \rho_j\rangle)\right)\right)^{\text{ss}}
		\end{align}
		of sections that determine stable quasimaps. Moreover, the isomorphism classes correspond to the $T$-orbits in this open subscheme with respect to the action induced by
		\begin{align*}
			T \curvearrowright H^0\left(\mathbb{P}^1, \mathcal{O}(\langle a, \rho_j\rangle)\right) \quad : \quad t\cdot s := \rho_j(t) s.
		\end{align*}
		It's immediate to check that since $\xi \in \chi(T)$ is such that the $T$-action is free on the semistable locus of $\mathbb{C}^m$, then the same is true when $\xi$ is thought as a linearisation for the newly defined $T$-action on (\ref{quasimapSections}). It's possible to show that (\ref{quasimapSectionsSS}) is precisely the semistable locus for the action with this choice of linearisation $\xi$. This shows that the quotient $L(a)/\!/T$ is a coarse moduli space for the quasimap functor. It's easy to see that this space comes with a universal family of quasimaps, namely
		\begin{equation}\label{toricUniversalQmap}
			\begin{tikzcd}
				L(a)^{\text{ss}} \times_{\mathbb{C}^\ast} \mathcal{O}(-1) \arrow[r, "u"] \arrow[d, "\pi"] & \mathbb{C}^m\\
				L(a)/\!/T \times \mathbb{P}^1 &
			\end{tikzcd}
		\end{equation}
		where $\mathbb{C}^\ast$ acts on $L(a)^{\text{ss}}$ by mapping to $T$ via the cocharacter $a$, or in other words
		\begin{align*}
			\mathbb{C}^\ast \curvearrowright H^0\left(\mathbb{P}^1, \mathcal{O}(\langle a, \rho_j\rangle)\right) \quad : \quad \tau\cdot s := \tau^{\langle a, \rho_j\rangle} s.
		\end{align*}
		so that the morphism $u$ is given by the evaluation $[s, (X,Y)] := s(X,Y)$ of the section, thought as a function on the total space of $\mathcal{O}(-1)$. Notice that the dimension of the quasimap space $L(a)/\!/T$ is often greater than the expected (or virtual) dimension of the moduli space of maps from $\mathbb{P}^1$ of degree $a$:
		\begin{align*}
			\text{dim}(L(a)/\!/T) &= \sum_{j\vert \langle a, \rho_j \rangle \geq 0} (1+ \langle a, \rho_j\rangle) - n\\
			&\geq \sum_{j=1}^m (1+ \langle a, \rho_j\rangle) - n = \text{virtual dim}(L(a)/\!/T).
		\end{align*}
		It's easy to show from its definition that the virtual fundamental class is
		\begin{align*}
			[L(a)/\!/T]^\text{vir} = \text{CW}^{a}\left(\prod_{\substack{j \,\,\vert \,\, \langle a, \rho_j \rangle < 0}}\rho_j^{-1-\langle a, \rho_j \rangle}\right) \cap [L(a)/\!/T],
		\end{align*}
		and it takes care of this difference in the dimensions.
	\end{ex}
	\subsection{Equivariant quasimap invariants}\label{SectionEquivariantQuasimapInvariants}
	The moduli space $\text{Q}(V/\!/G, \delta)$ inherits a $\mathbb{C}^\ast$-action from the action $s\cdot[x_0 : x_1]:= [sx_0,x_1]$ on the domain $\mathbb{P}^1$: given a quasimap, the action of $t \in \mathbb{C}^\ast$ on it can be described as follows:
	\[ 
	t\cdot 
	\begin{tikzcd}
		P \arrow[r, "\varphi"]\arrow[d, "\pi"] & V\\
		\mathbb{P}^1 & 
	\end{tikzcd}
	= 
	\begin{tikzcd}
		P \arrow[r, "\varphi"]\arrow[d, "t \cdot \pi"] & V\\
		\mathbb{P}^1 & 
	\end{tikzcd}.
	\]
	This description extends immediately to families of quasimaps, and by the universal property of the quasimap space it lifts to a $\mathbb{C}^\ast$-action on the universal family $\mathcal{P}_\delta$ which makes the morphism $u : \mathcal{P}_\delta \rightarrow V$ invariant. By restricting to a fixed point $p \in \lbrace 0, \infty \rbrace \in \mathbb{P}^1$ we find a $\mathbb{C}^\ast$-equivariant structure on $\mathcal{P}_{\delta \vert p}$ and the map $u_p$ in the diagram (\ref{CWdiagram}) is invariant.
	As in the nonequivariant case we can define the Chern-Weil homomorphism
	\begin{align}\label{CWeq}
		CW^\delta_p : A^\ast_{G\times \mathbb{C}^\ast}(\text{pt}) \rightarrow A^\ast_{G \times \mathbb{C}^\ast}(\mathcal{P}_{\delta\vert p}) \xrightarrow{d} A^\ast_{\mathbb{C}^\ast}(\text{Q}(V/\!/G, \delta)).
	\end{align}
	Notice that, unlike the nonequivariant case, the choice of $p \in \lbrace 0, \infty \rbrace$ yields two different morphisms $\text{CW}^\delta_0 \neq \text{CW}^\delta_\infty$. We will fix the choice of $p=\infty$ and work with $\text{CW}^\delta:= \text{CW}^\delta_\infty$.
	\begin{definition}
		Let $P \in A_G^\ast(\text{pt})$. The $\mathbb{C}^\ast$-equivariant \textit{graph-quasimap invariant} of $V/\!/G$ of class $\delta$ with insertion $P$ is the number
		\begin{align*}
			\int_{[Q(V/\!/G, \delta)]^{\text{vir}}} \text{CW}^\delta(P) \in \mathbb{Q}[z],
		\end{align*}
		where $[Q(V/\!/G,\delta)]^{\text{vir}}$ is the equivariant lift of the virtual fundamental class described in \cite[Section 7.3]{KimCiocanMaulik} and $z$ is the equivariant variable for the $\mathbb{C}^\ast$-action.
	\end{definition} 
	\begin{ex}\label{exampleToricqmaps2}
		We continue the example of the previous section by enhancing the picture with the $\mathbb{C}^\ast$-action we just defined. The action on the quasimap moduli space comes from the action on $L(a)$ given by
		\begin{align}\label{actionOnLc}
			\mathbb{C}^\ast \curvearrowright H^0(\mathbb{P}^1, \mathcal{O}(\langle a, \rho_j\rangle)) \quad : \quad s \cdot X^\alpha Y^\beta := s^\beta X^\alpha Y^\beta.
		\end{align}
		If we lift the action from $\mathbb{P}^1$ to the tautological $\mathbb{C}^\ast$-bundle $\mathcal{O}(-1)$ by setting $s\cdot(X,Y):= (X, s^{-1}Y)$ we obtain an equivariant structure on the universal quasimap (\ref{toricUniversalQmap}) so that $u$ is invariant. By restricting the universal quasimap to $L(a)/\!/T \times \lbrace \infty \rbrace$ we obtain the diagram
		\[\begin{tikzcd}
			L(a)^\text{ss} \arrow[r, "u_\infty"] \arrow[d, "\pi"] & \mathbb{C}^m\\
			L(a)/\!/T &
		\end{tikzcd}\]
		where the map $u_\infty$ is simply the evaluation at $(1,0)$.
	\end{ex}
	\subsection{The fixed loci}\label{SectionTheIFunction}
	The key ingredient that makes equivariant invariants easier to compute is $\mathbb{C}^\ast$-localisation. 
	A good application of this technique always starts with a description of the fixed locus $Q(V/\!/G, \delta)^{\mathbb{C}^\ast}$. Given $p \in \lbrace 0 , \infty \rbrace$ consider the substack $Q(V/\!/G, \delta)_p$ parameterising quasimaps with only $p$ as a basepoint and such that the induced map to $V/\!/G$ is constant. To each of these quasimaps one can associate the corresponding point in $V/\!/G$ obtaining an evaluation map
	\begin{align*}
		\text{ev}:Q(V/\!/G, \delta)_p \rightarrow V/\!/G.
	\end{align*}
	The fixed locus of the $\mathbb{C}^\ast$-action is described by the following result \cite[Section 4]{KimCiocanWallCrossing}
	\begin{pro}\label{SplitFixedLocus}
		Let $\delta_1, \delta_2 \in \chi(G)^\vee$ be so that $\delta = \delta_1 + \delta_2$. The fiber product
		\begin{align}\label{fixedLocusSplits}
			F_{\delta_1, \delta_2} := Q(V/\!/G, \delta_1)_0 \times_{V/\!/G} Q(V/\!/G, \delta_2)_\infty 
		\end{align}
		over the evaluation maps admits a closed embedding into $Q(V/\!/G, \delta)$ as a fixed subvariety. Moreover, all the fixed quasimaps are of this form:
		\begin{align*}
			Q(V/\!/G, \delta)^{\mathbb{C}^\ast} = \coprod_{\delta_1+ \delta_2 = \delta} F_{\delta_1, \delta_2}.
		\end{align*}
	\end{pro}
	Notice that, in our notation, $F_{\delta_1, 0} \simeq Q(V/\!/G, \delta_1)_0$.
	\begin{rem}\label{defineInjection}
		Proposition \ref{SplitFixedLocus} shows the existence of an injection of $F_{\delta_1,\delta_2}$ into $Q(V/\!/G,\delta)$, but it will be important for us to know how this is defined, namely how the universal quasimap pulls back from $Q(V/\!/G,\delta)$ to $F_{\delta_1, \delta_2}$. By definition we have the following fibered diagram
		\begin{equation}\label{fixed_fiber_diagram}
			\begin{tikzcd}
				F_{\delta_1, \delta_2} \arrow[rd, phantom, "\square"] \arrow[r, "p_2"] \arrow[d, "p_1"] & F_{\delta_1, 0} \arrow[d, "\text{ev}"] \\
				F_{0, \delta_2} \arrow[r, "\text{ev}"] & V/\!/G.
			\end{tikzcd}
		\end{equation}
		Consider the universal quasimaps $u_1 : \mathcal{P}_{\delta_1} \rightarrow V$ and $u_2 : \mathcal{P}_{\delta_2} \rightarrow V$ on $F_{\delta_1, 0}\times\mathbb{P}^1$ and $F_{0,\delta_2}\times\mathbb{P}^1$ respectively. They restrict, over $\mathbb{C}^\ast \subset \mathbb{P}^1$, to isomorphisms of principal bundles
		\begin{align*}
			\mathcal{P}_{\delta_1\vert F_{\delta_1, 0}\times \mathbb{C}^\ast} \xrightarrow{u_1} \text{ev}^\ast V(G)^\text{ss} \quad , \quad \mathcal{P}_{\delta_2\vert F_{0,\delta_2}\times \mathbb{C}^\ast} \xrightarrow{u_2} \text{ev}^\ast V(G)^\text{ss},
		\end{align*}
		where $V(G)^\text{ss} \rightarrow V/\!/G$ is thought as a $G$-principal bundle. By pulling back onto the fiber product $F_{\delta_1, \delta_2}$, we can use these isomorphisms to identify the quasimaps on these open subsets and, thus, to glue the quasimaps
		\begin{align*}
			p_1^\ast \mathcal{P}_{\delta_1\vert F_{\delta_1, \delta_2}\times \mathbb{C}_0} \xrightarrow{u_1} V \quad , \quad p_2^\ast \mathcal{P}_{\delta_2\vert F_{\delta_1,\delta_2}\times \mathbb{C}_\infty} \xrightarrow{u_2} V
		\end{align*}
		to a global quasimap $\mathcal{P} \rightarrow V$ on $F_{\delta_1, \delta_2} \times \mathbb{P}^1$, which defines an embedding into $Q(V/\!/G,\delta)$.
	\end{rem}
	\subsubsection{Virtual classes}
	The localised virtual fundamental classes of these fixed loci respect the product structure, as explained in \cite[Section 4]{KimCiocanWallCrossing}.
	\begin{pro}\label{splitting_virtual_classes}
		The localised virtual fundamental classes satisfy
		\begin{align*}
			\frac{[F_{\delta_1, \delta_2}]^{\text{vir}}}{e^{\mathbb{C}^\ast}(N^\text{vir}_{F_{\delta_1, \delta_2}})} = \frac{[F_{\delta_1, 0}]^{\text{vir}}}{e^{\mathbb{C}^\ast}(N^\text{vir}_{F_{\delta_1, 0}})} \times_{V/\!/G} \frac{[F_{0, \delta_2}]^{\text{vir}}}{e^{\mathbb{C}^\ast}(N^\text{vir}_{F_{0, \delta_2}})},
		\end{align*}
		where $\times_{V/\!/G}$ denotes the Gysin pullback along the diagonal of $V/\!/G$.
	\end{pro}
	The only part of this result that doesn't appear explicitly in \cite{KimCiocanWallCrossing} is the statement $[F_{\delta_1, \delta_2}]^{\text{vir}} = [F_{\delta_1, 0}]^{\text{vir}}\times_{V/\!/G}[F_{0, \delta_2}]^{\text{vir}}$. 
	This follows by an analysis of the obstruction theory on the quasimap space, whose virtual tangent bundle is $R {\pi}_\ast$ of the complex of $\mathbb{C}^\ast$-equivariant bundles on $Q(V/\!/G,\delta)\times \mathbb{P}^1$ 
	\begin{align}\label{obsTheory}
		\left[\frac{\mathcal{P}_\delta \times \mathfrak{g}}{G}\rightarrow \frac{\mathcal{P}_\delta \times V}{G}\right],
	\end{align}
	where both quotients are taken with respect to the diagonal $G$-action induced by the left action on each factor.
	We will prove Proposition \ref{splitting_virtual_classes} with a sequence of lemmas.
	\begin{lem}
		Given the fixed locus $i : F_{\delta_1, \delta_2} \hookrightarrow Q(V/\!/G, \delta)$, consider the restriction of the morphism (\ref{obsTheory}) to $F_{\delta_1, \delta_2} \times \mathbb{P}^1$, and let $\mathcal{Q}$ be the cokernel. Then the virtual tangent bundle of the quasimap space satisfies
		\begin{align*}
			Li^\ast T^{\text{vir}} = R\pi_\ast \mathcal{Q},
		\end{align*}
		where $\pi: F_{\delta_1, \delta_2} \times \mathbb{P}^1 \rightarrow F_{\delta_1, \delta_2}$ is the projection on the first component.
	\end{lem}
	\begin{proof}	
		Since $G$ acts freely on the semistable locus $V(G)^\text{ss}$ and the universal map $u : \mathcal{P}\rightarrow V$ sends the generic fiber into this semistable locus, the morphism (\ref{obsTheory}) is injective and we denote its cokernel with $\overline{\mathcal{Q}}$. The virtual tangent bundle on the quasimap space is the complex $T^{\text{vir}} = R\pi_\ast \overline{\mathcal{Q}}$. When restricted to $F_{\delta_1, \delta_2}$, since the projection $\pi$ is flat and proper we find by basechange \cite[\href{https://stacks.math.columbia.edu/tag/08IB}{Tag 08IB}]{stacks-project} with respect to the diagram
		\[\begin{tikzcd}
			F_{\delta_1, \delta_2} \times \mathbb{P}^1 \arrow[rd, phantom, "\square"] \arrow[r, "i \times \mathbb{1}"] \arrow[d, "\pi"] & Q(V/\!/G,\delta) \times \mathbb{P}^1 \arrow[d, "\pi"]\\
			F_{\delta_1, \delta_2} \times \mathbb{P}^1 \arrow[r, "i"] & Q(V/\!/G,\delta)
		\end{tikzcd}\]
		that
		\begin{align*}
			Li^\ast T^{\text{vir}} = R\pi_\ast L(i\times \mathbb{1})^\ast \overline{\mathcal{Q}} = R\pi_\ast \left(\overline{\mathcal{Q}}_{\vert F_{\delta_1, \delta_2}\times \mathbb{P}^1}\right) = R\pi_\ast \mathcal{Q},
		\end{align*}
		where in the second equality we have used that the morphism (\ref{obsTheory}) remains injective after restriction to $F_{\delta_1 , \delta_2} \times \mathbb{P}^1$. 
	\end{proof}
	We now want to compute this complex by Mayer-Vietoris relative to the covering of $\mathbb{P}^1$ given by $\mathbb{C}_0$ and $\mathbb{C}_\infty$. Set $U_0 := F_{\delta_1, \delta_2}\times \mathbb{C}_0$, $U_\infty := F_{\delta_1, \delta_2}\times \mathbb{C}_\infty$ and $U_{0 \infty}:= U_0 \cap U_\infty$. We keep track of the notation using the commutative diagram
	\[\begin{tikzcd}
		& U_{0\infty} \arrow[ld, hook] \arrow[d, hook] \arrow[rd, hook]&\\
		U_0 \arrow[r, hook] \arrow[dr, swap, "\pi_0"] & F_{\delta_1, \delta_2}\times \mathbb{P}^1 \arrow[d, "\pi"]& U_\infty \arrow[l, hook] \arrow[ld, "\pi_\infty"]\\
		& F_{\delta_1, \delta_2}& 
	\end{tikzcd}\]
	and denote with $\pi_{0\infty} : U_{0\infty}\rightarrow F_{\delta_1, \delta_2}$ the composition from the top to the bottom.
	\begin{lem}\label{invariant_cech}
		The $\mathbb{C}^\ast$-fixed part of the restriction $R\pi_\ast \mathcal{Q}$ of the virtual tangent bundle of the quasimap space fits in the exact triangle
		\begin{align}
			(R\pi_\ast \mathcal{Q})^{\mathbb{C}^\ast} \rightarrow (R{\pi_0}_\ast \mathcal{Q}_{\vert U_{0}})^{\mathbb{C}^\ast} \oplus (R{\pi_\infty}_\ast \mathcal{Q}_{\vert U_{\infty}})^{\mathbb{C}^\ast} \rightarrow \text{ev}^\ast T_{V/\!/G} \xrightarrow{+1} \cdots
		\end{align}
		in $D_{\text{QCoh}}(F_{\delta_1, \delta_2})$.
	\end{lem}
	\begin{proof}
		By a standard Mayer-Vietoris computation \cite[\href{https://stacks.math.columbia.edu/tag/08JK}{Tag 08JK}]{stacks-project}, $R\pi_\ast \mathcal{Q}$ fits in the exact triangle
		\begin{align*}
			R\pi_\ast \mathcal{Q} \rightarrow R{\pi_0}_\ast \mathcal{Q}_{\vert U_{0}} \oplus R{\pi_\infty}_\ast \mathcal{Q}_{\vert U_{\infty}} \rightarrow R{\pi_{0\infty}}_\ast \mathcal{Q}_{\vert U_{0\infty}} \xrightarrow{+1} \cdots.
		\end{align*}
		Since on $V(G)^\text{ss}$ we have the short exact sequence 
		\begin{align*}
			0 \rightarrow \underline{\mathfrak{g}} \rightarrow T_V \rightarrow \pi^\ast T_{V/\!/G} \rightarrow 0,
		\end{align*} 
		there is a canonical isomorphism $\mathcal{Q}_{\vert U_{0\infty}} \simeq \pi_{0\infty}^\ast \text{ev}^\ast T_{V/\!/G}$, and we can rewrite the last term in the triangle by using the projection formula as
		\begin{align}\label{cech_complex}
			R\pi_\ast \mathcal{Q} \rightarrow R{\pi_0}_\ast \mathcal{Q}_{\vert U_{0}} \oplus R{\pi_\infty}_\ast \mathcal{Q}_{\vert U_{\infty}} \rightarrow \text{ev}^\ast T_{V/\!/G} \otimes R {\pi_{0\infty}}_\ast \mathcal{O}_{U_{0\infty}} \xrightarrow{+1} \cdots.
		\end{align}
		By taking $\mathbb{C}^\ast$-invariants we end up with the claimed exact triangle.
		\end{proof}
		\begin{lem}\label{derived_identifications}
			Let $\mathcal{Q}_0$ be the restriction of the cokernel of the morphism (\ref{obsTheory}) on $Q(V/\!/G,\delta_1) \times \mathbb{P}^1$ to $F_{\delta_1, 0}\times \mathbb{P}^1$ and, analogously, let $\mathcal{Q}_\infty$ be the restriction of the cokernel of the morphism (\ref{obsTheory}) on $Q(V/\!/G,\delta_2) \times \mathbb{P}^1$ to $F_{0,\delta_2}\times \mathbb{P}^1$. The following isomorphisms hold true in the derived category of $F_{\delta_1, \delta_2}$:
			\begin{align*}
				\left(R{\pi_0}_\ast \mathcal{Q}_{\vert U_0}\right)^{\mathbb{C}^\ast} &\simeq \left(R\pi_\ast L p_1^\ast \mathcal{Q}_0\right)^{\mathbb{C}^\ast}\\
				\left(R{\pi_\infty}_\ast \mathcal{Q}_{\vert U_\infty}\right)^{\mathbb{C}^\ast} &\simeq \left(R\pi_\ast L p_2^\ast \mathcal{Q}_\infty\right)^{\mathbb{C}^\ast}.
			\end{align*}
		\end{lem}
		\begin{proof}
			We will prove the first equality as the second one is completely analogous.
			First of all notice that the locally free resolution (\ref{obsTheory}) of $\mathcal{Q}_0$ stays injective after pulling back by $p_1$, hence $L p_1^\ast \mathcal{Q}_0 = p_1^\ast \mathcal{Q}_0 =: Q_0$.
			By Remark \ref{defineInjection} we have the identifications
		\begin{align*}
			(Q_0)_{\vert U_0} \simeq \mathcal{Q}_{\vert U_0}  \qquad \text{and} \qquad (Q_0)_{\vert U_\infty} \simeq \pi_\infty^\ast \text{ev}^\ast T_{V/\!/G}.
		\end{align*}
		This implies that the Mayer-Vietoris for $R \pi_\ast Q_0$ fits it in the triangle
		\begin{align*}
			R\pi_\ast Q_0 &\rightarrow R{\pi_0}_\ast \mathcal{Q}_{\vert U_0} \oplus \left(\text{ev}^\ast T_{V/\!/G} \otimes R{\pi_\infty}_\ast \mathcal{O}_{U_{\infty}}\right)\rightarrow \\ 
			&\rightarrow \text{ev}^\ast T_{V/\!/G} \otimes R {\pi_{0\infty}}_\ast \mathcal{O}_{U_{0\infty}} \xrightarrow{+1} \cdots.
		\end{align*}
		Taking the fixed parts we find
		\begin{align*}
			\left(R\pi_\ast Q_0\right)^{\mathbb{C}^\ast} &\rightarrow \left(R{\pi_0}_\ast \mathcal{Q}_{\vert U_0}\right)^{\mathbb{C}^\ast} \oplus \text{ev}^\ast T_{V/\!/G} \rightarrow \text{ev}^\ast T_{V/\!/G} \xrightarrow{+1} \cdots,
		\end{align*}
		which shows that we have an isomorphism $\left(R\pi_\ast Q_0\right)^{\mathbb{C}^\ast} \simeq \left(R{\pi_0}_\ast \mathcal{Q}_{\vert U_0}\right)^{\mathbb{C}^\ast}$.
	\end{proof}
	\begin{proof}[Proof of Proposition \ref{splitting_virtual_classes}]
		Combining Lemma \ref{invariant_cech} and Lemma \ref{derived_identifications}, we find the exact triangle
		\begin{align}\label{final_exact_triangle}
			\left(R\pi_\ast \mathcal{Q}\right)^{\mathbb{C}^\ast} \rightarrow \left(R\pi_\ast Q_0\right)^{\mathbb{C}^\ast} \oplus \left(R\pi_\ast Q_\infty\right)^{\mathbb{C}^\ast} \rightarrow \text{ev}^\ast T_{V/\!/G} \xrightarrow{+1} \cdots.
		\end{align}
		We see that
		\begin{align*}
			R\pi_\ast Q_0 \simeq Lp_1^\ast R\pi_{0 \ast} \mathcal{Q}_0
		\end{align*}
		by basechange \cite[\href{https://stacks.math.columbia.edu/tag/08IB}{Tag 08IB}]{stacks-project} with respect to the diagram
		\[\begin{tikzcd}
			F_{\delta_1,\delta_2}\times \mathbb{P}_1 \arrow[rd, phantom, "\square"] \arrow[r, "p_1 \times \mathbb{1}"]\arrow[d, "\pi"] & F_{\delta_1,0} \times \mathbb{P}^1 \arrow[d, "\pi_0"]\\
			F_{\delta_1,\delta_2} \arrow[r, "p_1"] & F_{\delta_1,0}
		\end{tikzcd}\]
		since $\pi_0$ is flat and proper, and the same holds true for $Q_\infty$.
		This shows that (\ref{final_exact_triangle}) is precisely the exact triangle
		\begin{align*}
			T_{F_{\delta_1, \delta_2}}^\text{vir} \rightarrow T_{F_{\delta_1, 0}}^\text{vir} \boxplus T_{F_{0, \delta_2}}^\text{vir} \rightarrow \text{ev}^\ast T_{V/\!/G} \xrightarrow{+1} \cdots
		\end{align*}
		we need to obtain the claimed equality of virtual fundamental classes by combining Propositions 7.4 and 7.5 from \cite{Behrend}.
	\end{proof}
	\subsubsection{The Chern-Weil morphism on fixed loci}
	We have seen how the virtual classes of the fixed loci split with respect to (\ref{fixedLocusSplits}). What happens to the Chern-Weil morphism? 
	\begin{pro}\label{splittingOfCW}
		Let the equality $\delta = \delta_1 + \delta_2$ hold true in $\chi(G)^\vee$ and consider the fixed loci $F_{\delta_1, \delta_2}\subset Q(V/\!/G,\delta)$ and $F_{0, \delta_2}\subset Q(V/\!/G,\delta_2)$.
		They are related by the projection $p_2 : F_{\delta_1, \delta_2}\rightarrow F_{0, \delta_2}$ given by (\ref{fixedLocusSplits}). Then
		\begin{align*}
			\text{CW}^\delta(P)_{\vert F_{\delta_1, \delta_2}} = p_2^\ast \text{CW}^{\delta_2}(P)_{\vert F_{0, \delta_2}}
		\end{align*}
		for every $P \in A^\ast_G(\text{pt})$.
	\end{pro}
	\begin{proof}
		\sloppy Consider the universal families $\mathcal{P}_{\delta}$ on $Q(V/\!/G, \delta)\times \mathbb{P}^1$ and $\mathcal{P}_{\delta_2}$ on $Q(V/\!/G, \delta_2)\times \mathbb{P}^1$. The thesis follows by the isomorphism $\mathcal{P}_{\delta\vert F_{0,d_2}\times \mathbb{C}_\infty} \simeq p_2^\ast \mathcal{P}_{\delta_2 \vert F_{d_1,d_2}\times\mathbb{C}_\infty}$ described in Remark \ref{defineInjection}.
	\end{proof}
	We have the following immediate corollary, that will be crucial for computing our quasimap integrals:
	\begin{cor}\label{splitting_contributions}
		Consider the fixed locus $F_{\delta_1, \delta_2}$ in $Q(V/\!/G, \delta)$ and a class $P \in A^\ast_G(\text{pt})$. Then the pushforward to $V/\!/G$ through the evaluation map
		\begin{align*}
			\text{ev}_\ast\left(\text{CW}^\delta(P)_{\vert F_{\delta_1, \delta_2}} \cap \frac{[F_{\delta_1, \delta_2}]^\text{vir}}{e^{\mathbb{C}^\ast}(N^\text{vir}_{F_{\delta_1, \delta_2}})}\right)
		\end{align*}
		coincides with the intersection product in $V/\!/G$ of the pushforwards
		\begin{align*}
			\text{ev}_\ast\left(\frac{[F_{\delta_1, 0}]^\text{vir}}{e^{\mathbb{C}^\ast}(N^\text{vir}_{F_{\delta_1,0}})}\right) \cdot \text{ev}_\ast\left(\text{CW}^{\delta_2}(P)_{\vert F_{0, \delta_2}} \cap \frac{[F_{0, \delta_2}]^\text{vir}}{e^{\mathbb{C}^\ast}(N^\text{vir}_{F_{0, \delta_2}})}\right)
		\end{align*}
		through the evaluation morphisms of $F_{\delta_1,0}$ and $F_{0,\delta_2}$.
	\end{cor}
	\begin{proof}
		Consider the fibered diagram
		\[\begin{tikzcd}
			F_{\delta_1, \delta_2}\arrow[rd, phantom, "\square"] \arrow[d, "\text{ev}"] \arrow[r, "p_1 \times p_2"] & F_{\delta_1, 0} \times F_{0, \delta_2} \arrow[d, "\text{ev}\times \text{ev}"]\\
			V/\!/G \arrow[r, "\Delta"] & V/\!/G \times V/\!/G
		\end{tikzcd}\]
		and recall that Proposition \ref{splitting_virtual_classes} ensures that the localised virtual fundamental class of $F_{\delta_1, \delta_2}$ is the refined Gysin pullback $\Delta^!$ of the localised virtual classes of $F_{\delta_1, 0}$ and  $F_{0, \delta_2}$. By Proposition \ref{splittingOfCW} and compatibility of cohomology classes with Gysin morphisms (\cite[Property (C3), Section 17.1]{FultonInt}) we can write 
		\begin{align*}
			&\text{CW}^\delta(P)_{\vert F_{\delta_1, \delta_2}} \cap \frac{[F_{\delta_1, \delta_2}]^\text{vir}}{e^{\mathbb{C}^\ast}(N^\text{vir}_{F_{\delta_1, \delta_2}})}\\ = &\Delta^! \left(\frac{[F_{\delta_1, 0}]^\text{vir}}{e^{\mathbb{C}^\ast}(N^\text{vir}_{F_{\delta_1,0}})}, \text{CW}^{\delta_2}(P)_{\vert F_{0, \delta_2}} \cap \frac{[F_{0, \delta_2}]^\text{vir}}{e^{\mathbb{C}^\ast}(N^\text{vir}_{F_{0, \delta_2}})}\right).
		\end{align*}
		By \cite[Theorem 6.2 (a)]{FultonInt}, for all classes $\alpha \in A_\ast(F_{\delta_1, 0})$ and $\beta \in A_\ast(F_{0,\delta_2})$, we have 
		\begin{align*}
			\text{ev}_\ast \Delta^!(\alpha, \beta) = \Delta^\ast\left(\text{ev}_\ast \alpha, \text{ev}_\ast \beta\right),
		\end{align*}
		and the thesis follows from applying this equality to the right-hand side of the previous equation.
	\end{proof}
	\subsubsection{The I-function}
	Here we recall the classical definition of the quasimap small $I$-function
	\begin{definition}
		For every $\delta \in \chi(G)^\vee$, the \textit{$\delta$-term of the $I$-function of $V/\!/G$} is the element $I_\delta^{V/\!/G} \in A_\ast(V/\!/G)[z]_z$ given by
		\begin{align*}
			I_\delta^{V/\!/G}(z):= \text{ev}_\ast \left(\frac{[F_{\delta, 0}]^\text{vir}}{e^{\mathbb{C}^\ast}(N^{\text{vir}}_{F_{\delta, 0}})}\right).
		\end{align*}
		The \textit{quasimap small $I$-function} of $V/\!/G$ is the formal sum
		\begin{align}\label{Ifuction}
			I^{V/\!/G}:= \sum_{\delta \in \chi(G)^\vee} q^\delta I^{V/\!/G}_\delta.
		\end{align}
		Since this is the only I-function we will consider in this paper, we will often refer to it as \textit{the I-function}.
	\end{definition}
	Notice that the quasimap $I$-function only encodes data coming from the $(\delta, 0)$-part of the fixed locus. The part indexed by $(0,\delta)$ is also related to the $I$-function in a simple way:
	\begin{align*}
		I_\delta^{V/\!/G}(-z) = \text{ev}_\ast \left(\frac{[F_{0,\delta}]^\text{vir}}{e^{\mathbb{C}^\ast}(N^{\text{vir}}_{F_{0,\delta}})}\right).
	\end{align*}
	This is easy to see by using the automorphism of $\mathbb{P}^1$ exchanging the coordinates to produce an automorphism of $Q(V/\!/G, \delta)$ which exchanges the action of $t$ with the action of $t^{-1}$. 
	It is a well known fact that classes coming from the remaining fixed loci can still be expressed in terms of the $I$-function. We can recover this by setting $P=1$ in Corollary \ref{splitting_contributions}:
	\begin{pro}\label{EvaluatingVirtualCycle}
		Let $\delta_1 + \delta_2 = \delta$. Then
		\begin{align*}
			\text{ev}_\ast \left( \frac{[F_{\delta_1, \delta_2}]^{\text{vir}}}{e^{\mathbb{C}^\ast}(N^{\text{vir}}_{F_{\delta_1, \delta_2}})} \right) = I^{V/\!/G}_{\delta_1}(z) I^{V/\!/G}_{\delta_2}(-z).
		\end{align*}
	\end{pro}
	\subsection{The Chern-Weil homomorphism in the toric case}
	It turns out that, in the toric case, the restriction of the Chern-Weil homomorphism to the fixed loci $F_{\tilde{\delta}_1, \tilde{\delta}_2}$ is related to the pullback through the evaluation maps. We start by explicitly describing the fixed loci for quasimaps with toric targets.
	\begin{ex}\label{exampleToricqmaps3}
		Given a cocharacter $a \in \chi(T)^\vee$ and a point $p \in \lbrace 0, \infty \rbrace \subset \mathbb{P}^1$ we denote with
		\begin{align*}
			L(a)_p \subset \oplus_{j=1}^m H^0\left(\mathbb{P}^1, \mathcal{O}(\langle a, \rho_j\rangle)\right)
		\end{align*}
		the subscheme formed by $\mathbb{C}^\ast$-fixed quasimaps that meet the unstable locus $(\mathbb{C}^m)^{\text{unst}}$ only at the point $p$. Notice that
		\begin{align*}
			L(a)_0 \simeq L(a)_\infty
		\end{align*}
		by exchanging the role of the variables $X,Y$ on $\mathcal{O}(-1)$. We now show that an element of $L(a)_0$ is necessarily of the form
		\begin{align*}
			\varphi(X,Y) = (k_1 X^{\langle a,\rho_1\rangle}, \dots, k_m X^{\langle a,\rho_m\rangle}).
		\end{align*}
		for some $k_1, ..., k_m \in \mathbb{C}$.
		Since $\varphi$ is a fixed quasimap with only poles at $0$, for every $X \neq 0$ and $Y\neq 0$ the image $\varphi(X,Y)$ belongs to the $T$-orbit of $\varphi(1,0)$. In particular, fixed $j \in \lbrace 1, \dots, m\rbrace$, $\varphi_{j}(X,Y)= 0$ if and only if $\varphi_{j}(1,0)= 0$. In case the $j^{th}$ component is nonzero, then $\varphi_{j}(X,Y)$ is a homogeneous polynomial of degree $\langle a,\rho_{j}\rangle$ that vanishes only at $0 \in \mathbb{P}^1$, hence there is a nonzero number $k_{j} \in \mathbb{C}^\ast$ so that $\varphi_{j}(X,Y) = k_{j} X^{\langle a, \rho_j\rangle}$, proving the claim.
		
		This shows that a quasimap belonging to $F_{a,b}$, where $a+b=c$, is a section of $\oplus_{j=1}^m \mathcal{O}(\langle c, \rho_j\rangle)$ of the form
		\begin{align*}
			\varphi_j(X,Y) = k_j X^{\langle a, \rho_j\rangle} Y^{\langle b, \rho_j\rangle}
		\end{align*}
		for some $k_1, ..., k_m \in \mathbb{C}$. 
		Now that we understand the shape of the quasimaps paramterised by $F_{a,b}$ we would like to lift the evaluation map 
		\begin{align*}
			\text{ev} : F_{a,b} \rightarrow \mathbb{C}^m/\!/T 
		\end{align*}
		to a $\mathbb{C}^\ast$-equivariant morphism $L(c)^{\text{ss}}_{\vert F_{a,b}} \rightarrow (\mathbb{C}^m)^\text{ss}$.
		If we restrict to the quasimaps lying above $F_{a,b}$, the $\mathbb{C}^\ast$-action (\ref{actionOnLc}) on $L(c)^{\text{ss}}$ restricts to $L(c)^{\text{ss}}_{\vert F_{a,b}}$ as
		\begin{align*}
			(s \cdot \varphi)_i = s^{\langle b,\rho_i\rangle} \varphi_i \qquad \forall i = 1, \dots, m.
		\end{align*}
		Consider the evaluation map $u_{(1,1)} : L(c)^{\text{ss}} \rightarrow \mathbb{C}^m$ at the point $(1,1) \in \mathcal{O}(-1)$. This lifts $\text{ev}$ and it becomes $\mathbb{C}^\ast$-equivariant once we define the $\mathbb{C}^\ast$-action on $\mathbb{C}^m$ as 
		\begin{align*}
			(s \cdot x)_j := s^{\langle b, \rho_j\rangle}x_j.
		\end{align*}
		With this equivariant structure on $\mathbb{C}^m$ we have a commutative, $\mathbb{C}^\ast$-equivariant diagram
		\begin{equation}\label{equivariantDiagram}
			\begin{tikzcd}
				L(c)^{\text{ss}}_{\vert F_{a,b}} \arrow[d, "\pi"] \arrow[rr,"u_{(1,1)}"] \arrow[rrd] \arrow[rdd]& & \mathbb{C}^m\\
				F_{a,b} \arrow[rd, swap, "\text{ev}"]& & (\mathbb{C}^m)^{\text{ss}}\arrow[u, hook, "i"] \arrow[ld, "q"]\\
				& \mathbb{C}^m/\!/T &
			\end{tikzcd}
		\end{equation}
		and notice that the residual $\mathbb{C}^\ast$-action on the quotient $\mathbb{C}^m/\!/T$ is trivial.
	\end{ex}
	Thanks to this simple description in the toric case we can prove the following relation between the Chern-Weil homomorphism and the evaluation map:
	\begin{lem}\label{ChernWeilEvaluation}
		Let $\tilde{\delta}_1, \tilde{\delta}_2\in \chi(T)^\vee$, set $\tilde{\delta} := \tilde{\delta}_1 + \tilde{\delta}_2$ and consider the quasimap space $Q(V/\!/T, \tilde{\delta})$. For every $\alpha \in \chi(T)$, consider the line bundle $\mathcal{L}_a$ on $V/\!/T$ induced from the equivariant bundle on $V$ defined by the representation $\mathbb{C}_\alpha$ (namely the fiber is $\mathbb{C}$ with $T$-action defined by $\alpha$). We have
		\begin{align*}
			\text{CW}^{\tilde{\delta}}\left(\alpha + \langle \tilde{\delta}_2, \alpha \rangle z\right)_{\vert F_{\tilde{\delta}_1 , \tilde{\delta}_2}} = \text{ev}^\ast c_1(\mathcal{L}_\alpha)
		\end{align*}
		in $A_{\mathbb{C}^\ast}^\ast(F_{\tilde{\delta}_1 , \tilde{\delta}_2})_z \simeq A^\ast(F_{\tilde{\delta}_1 , \tilde{\delta}_2})[z]_z$.
	\end{lem}
	\begin{proof}
		Recall the explicit description of quasimap moduli spaces to toric targets that we have discussed in Examples \ref{exampleToricqmaps1},  \ref{exampleToricqmaps2} and \ref{exampleToricqmaps3}.
		If we let $\mathbb{C}^\ast$ act on $V$ via the cocharacter $\tilde{\delta}_2$, the diagram (\ref{equivariantDiagram}) becomes a commutative $\mathbb{C}^\ast$-equivariant-diagram of the form
		\begin{equation*}
			\begin{tikzcd}
				\mathcal{P}_{\delta \vert F_{\tilde{\delta}_1, \tilde{\delta}_2}} \arrow[d, "\pi"] \arrow[rr,"u"] \arrow[rrd, "x"] \arrow[rdd, "y"]& & V\\
				F_{\tilde{\delta}_1, \tilde{\delta}_2} \arrow[rd, swap, "\text{ev}"]& & V^{\text{ss}}\arrow[u, hook, "i"] \arrow[ld, "q"]\\
				& V/\!/T &
			\end{tikzcd}
		\end{equation*}
		Given $\alpha \in \chi(T)$ we can consider the class $c_1(\mathcal{L}_\alpha) \in A^\ast(V/\!/T) \subset A^\ast_{\mathbb{C}^\ast}(V/\!/T)$. It is immediate to check that $q^\ast c_1(\mathcal{L}_\alpha) = i^\ast(\alpha + \langle\tilde{\delta}_2, \alpha \rangle z)$ by directly looking at the induced $\mathbb{C}^\ast$-equivariant structure on the descended bundle $\mathcal{L}_\alpha$. Then by functoriality
		\begin{align*}
			\pi^\ast \text{ev}^\ast c_1(\mathcal{L}_\alpha) &= y^\ast c_1(\mathcal{L}_\alpha) = x^\ast q^\ast c_1(\mathcal{L}_\alpha) \\
			&= x^\ast i^\ast \left(\alpha + \langle \tilde{\delta}_2, \alpha \rangle z\right)= u^\ast \left(\alpha + \langle \tilde{\delta}_2, \alpha \rangle z\right).
		\end{align*}
		By definition of the Chern-Weil homomorphism, this concludes the proof.
	\end{proof}
	\section{Abelianisation}\label{SectionAbelianisation}
	The goal of this section is to prove the abelianisation formula of Theorem \ref{MainTheoremAbFor}, relating quasimap invariants of $V/\!/G$ to quasimap invariants of $V/\!/T$. Throughout this section we will adopt the following notation, which will be handy in some computations:
	\begin{itemize}
		\item consider a ring $R$ and the corresponding ring $R[z]_z$ of Laurent polynomials. Given $P \in R[z]_z$, we will denote with $\overline{P} \in R[z]_z$ the Laurent polynomial defined as $\overline{P}(z):= P(-z)$.
	\end{itemize}
	\subsection{The key proposition}
	Fixed $\delta \in \chi(G)^\vee$ and a splitting $\delta = \delta_1 + \delta_2$ we can consider the fixed locus $F_{\delta_1, \delta_2}$ in the quasimap space $Q(V/\!/G, \delta)$. Analogously, if we pick $\tilde{\delta}_1, \tilde{\delta}_2 \in \chi(T)^\vee$, whose restrictions on $\chi(G)$ are $\delta_1, \delta_2$ respectively (in short-hand notation we denote this with $\tilde{\delta}_i \mapsto \delta_i$), we can consider the fixed locus $F_{\tilde{\delta}_1, \tilde{\delta}_2}$ in $Q(V/\!/T, \tilde{\delta})$ for $\tilde{\delta} := \tilde{\delta}_1+\tilde{\delta}_2$.
	Consider the diagram relating the quotients of $V$ by $T$ and $G$:
	\begin{equation}\label{abelianisationDiagram}
		\begin{tikzcd}
			& & F_{\tilde{\delta}_1, \tilde{\delta}_2} \arrow[d, "\text{ev}"]\\
			& V(G)^{\text{ss}}/\!/T \arrow[r, hook, "j"] \arrow[d, "g"]& V/\!/T\\
			F_{\delta_1, \delta_2} \arrow[r, "\text{ev}"] & V/\!/G &
		\end{tikzcd}.
	\end{equation}
	Here $j$ is an open embedding (since the $G$-semistable locus is open in the $T$-semistable one) and $g$ is a smooth morphism with fiber $G/T$ (since the actions are free on the semistable loci). 
	
	The key proposition for proving the abelianisation formula relates the pushforward of the virtual class from $F_{\delta_1, \delta_2}$ to the pushforward of the virtual classes from $F_{\tilde{\delta}_1, \tilde{\delta}_2}$, for all possible lifts $\tilde{\delta}_i \mapsto \delta_i$, when both classes are pulled back onto $V(G)^\text{ss}/\!/T$ via the maps $j$ and $g$.
	
	The discrepancy between these two classes will be given in terms of the following correction terms: given $\tilde{\delta} \in \chi(T)^\vee$ and $w \in \chi(T)$, we define the class $C(\tilde{\delta}, w) \in A^\ast(V/\!/T)[z]_z$ as
	\begin{align*}
		C(\tilde{\delta}, w):&= \frac{\prod_{k=-\infty}^{\langle \tilde{\delta}, w \rangle}(c_1(\mathcal{L}_w)+kz)}{\prod_{k=-\infty}^{0}(c_1(\mathcal{L}_w)+kz)}\\
        &= \begin{cases}
			\prod_{k=\langle \tilde{\delta}, w \rangle+1}^0 (c_1(\mathcal{L}_w)+kz)^{-1} & \text{if } \langle \tilde{\delta}, w \rangle < 0\\
			1 & \text{if } \langle \tilde{\delta}, w \rangle=0\\
			\prod_{k=1}^{\langle \tilde{\delta}, w \rangle} (c_1(\mathcal{L}_w)+kz) & \text{if } \langle \tilde{\delta}, w \rangle > 0
		\end{cases}.
	\end{align*}
	To be precise, we should stress that $C(\tilde{\delta}, w)$ is only well defined when $\langle\tilde{\delta}, w\rangle \geq 0$. On the other hand, see \cite[Section 5.4]{WebbAbelianNonAbelian} for the proof that \textit{the product} $\prod_{\alpha\in \Delta} C(\tilde{\delta}, \alpha)$ over the roots of $G$ is always a well defined element of $A^\ast(V/\!/T)[z]_z$. 
	\begin{definition}
		Given $\tilde{\delta} \in \chi(T)^\vee$ we define the class in $A^\ast(V/\!/T)[z]_z$
		\begin{align}\label{cClass}
			C(\tilde{\delta}) := \prod_{\alpha\in \Delta} C(\tilde{\delta}, \alpha) ,
		\end{align}
		where the product is over the set $\Delta$ of roots of $G$.
	\end{definition}
	The crucial result that implies the abelianisation formula is the following
	\begin{pro}\label{AbelianisationClasses}
		Let $P \in A^\ast_G(\text{pt})$, $\delta_1, \delta_2 \in \chi(G)^\vee$ and set $\delta := \delta_1 + \delta_2$. Then the class
		\begin{align*}
			g^\ast \text{ev}_\ast \left(\text{CW}^\delta(P)_{\vert F_{\delta_1, \delta_2}} \cap \frac{[F_{\delta_1, \delta_2}]^{\text{vir}}}{e^{\mathbb{C}^\ast}(N^\text{vir}_{F_{\delta_1, \delta_2}})}\right)
		\end{align*}
		in $A_\ast(V(G)^\text{ss}/\!/T)[z]_z$ coincides with
		\begin{align*}
			j^\ast \sum_{\substack{\tilde{\delta}_1\mapsto \delta_1\\\tilde{\delta}_2\mapsto \delta_2}} C(\tilde{\delta}_1)\overline{C(\tilde{\delta}_2)} \cap \text{ev}_\ast\left(\text{CW}^{\tilde{\delta}}(P)_{\vert F_{\tilde{\delta}_1, \tilde{\delta}_2}} \cap \frac{[F_{\tilde{\delta}_1, \tilde{\delta}_2}]^{\text{vir}}}{e^{\mathbb{C}^\ast}(N^\text{vir}_{F_{\tilde{\delta}_1, \tilde{\delta}_2}})}\right),
		\end{align*}
		where $\tilde{\delta}:= \tilde{\delta}_1+\tilde{\delta}_2$.
	\end{pro} 
	\subsubsection{Reduction to a simpler result}
	We want to prove that, if we manage to show the following simpler result, then Proposition \ref{AbelianisationClasses} follows from facts that we already know.
	\begin{lem}\label{generalisation_webb}
		Fix $P \in A^\ast_G(\text{pt})$ and $\delta \in \chi(G)^\vee$. Then the class
		\begin{align*}
			g^\ast \text{ev}_\ast \left(\text{CW}^\delta(P)_{\vert F_{0, \delta}} \cap \frac{[F_{0, \delta}]^{\text{vir}}}{e^{\mathbb{C}^\ast}(N^\text{vir}_{F_{0, \delta}})}\right)
		\end{align*}
		in $A_\ast(V(G)^\text{ss}/\!/T)[z]_z$ coincides with
		\begin{align*}
			j^\ast \sum_{\tilde{\delta}\mapsto \delta} \overline{C(\tilde{\delta})} \cap \text{ev}_\ast\left(\text{CW}^{\tilde{\delta}}(P)_{\vert F_{0, \tilde{\delta}}} \cap \frac{[F_{0, \tilde{\delta}}]^{\text{vir}}}{e^{\mathbb{C}^\ast}(N^\text{vir}_{F_{0, \tilde{\delta}}})}\right).
		\end{align*}
	\end{lem}
	Notice that if $P=1$ then this is nothing but \cite[Theorem 1.1.1]{WebbAbelianNonAbelian} for $F_{0,\delta}$ instead of $F_{\delta,0}$ (which just exchanges $z$ with $-z$ in the formulae). The same proof of \cite[Theorem 1.1.1]{WebbAbelianNonAbelian} goes through also for $P \neq 1$, for details see Appendix \ref{appendix_proof_key_proposition}.
	\begin{proof}[Proof of Proposition \ref{AbelianisationClasses}]
		We start by applying Corollary \ref{splitting_contributions} to both classes in the statement of Proposition \ref{AbelianisationClasses}, splitting the contributions of type $(\delta,0)$ from the ones of type $(0,\delta)$. For the former type the result follows from \cite[Theorem 1.1.1]{WebbAbelianNonAbelian}, while for the latter we just need Lemma \ref{generalisation_webb}.
	\end{proof}
	\subsection{The abelianisation formula}
	We are ready to prove the abelianisation formula for quasimap invariants. We now need to use this classical result, proven in the symplectic category by Martin \cite{Martin} and in the algebraic one by Maddock \cite{maddock2014ratio}:
	\begin{theorem}\label{MartinTheorem}
		Assume that $\alpha \in A^\ast(V/\!/G)$ and $\beta \in A^\ast(V/\!/T)$ are such that $g^\ast \alpha = j^\ast\beta$. Then 
		\begin{align*}
			\int_{V/\!/G} \alpha = \frac{1}{\vert W \vert} \int_{V/\!/T} \beta \prod_{\alpha \in \Delta} c_1(\mathcal{L}_{\alpha})
		\end{align*}
		where $\vert W\vert$ is the order of the Weyl group of $G$.
	\end{theorem}
	Once we define the class $D$ as in (\ref{Dfunction}), this allows to prove the following 
	\begin{lem}
		Let $\delta_1, \delta_2 \in \chi(G)^\vee$ and set $\delta:= \delta_1+ \delta_2$. For every $P \in A_G^\ast(\text{pt})$, the equality between the localised contribution to quasimap invariants
		\begin{align*}
			\int_{[F_{\delta_1, \delta_2}]^\text{vir}} \frac{\text{CW}^\delta(P)}{e^{\mathbb{C}^\ast}(N^\text{vir}_{F_{\delta_1, \delta_2}})} =  \frac{1}{\vert W \vert}\sum_{\substack{\tilde{\delta}_1 \mapsto \delta_1\\\tilde{\delta}_2 \mapsto \delta_2}} \int_{[F_{\tilde{\delta}_1, \tilde{\delta}_2}]^\text{vir}}\frac{\text{CW}^{\tilde{\delta}}\left(P \prod_{\alpha \in \Delta} D(\tilde{\delta}, \alpha)\right)}{e^{\mathbb{C}^\ast}(N^\text{vir}_{F_{\tilde{\delta}_1, \tilde{\delta}_2}})}
		\end{align*}
		holds true.
	\end{lem}
	\begin{proof}
		By applying Martin's formula together with the equality of Proposition \ref{AbelianisationClasses} we obtain that the localised contribution
		\begin{align}\label{FixedContr}
			\int_{[F_{\delta_1, \delta_2}]^\text{vir}} \frac{\text{CW}^\delta(P)}{e^{\mathbb{C}^\ast}(N^\text{vir}_{F_{\delta_1, \delta_2}})}
		\end{align}
		is equal to 
		\begin{align*}
			\frac{1}{\vert W \vert}\sum_{\substack{\tilde{\delta}_1 \mapsto \delta_1\\\tilde{\delta}_2 \mapsto \delta_2}} \int_{[F_{\tilde{\delta}_1, \tilde{\delta}_2}]^\text{vir}}\frac{\text{CW}^{\tilde{\delta}}(P)}{e^{\mathbb{C}^\ast}(N^\text{vir}_{F_{\tilde{\delta}_1, \tilde{\delta}_2}})} \text{ev}^\ast \left(\prod_{\alpha\in \Delta}C(\tilde{\delta}_1, \alpha)\overline{C(\tilde{\delta}_2, \alpha)}c_1(\mathcal{L}_\alpha)\right).
		\end{align*}
		By Lemma \ref{ChernWeilEvaluation} we have
		\begin{align*}
			\text{ev}^\ast \left(\prod_{\alpha\in \Delta}C(\tilde{\delta}_1, \alpha)\overline{C(\tilde{\delta}_2, \alpha)}c_1(\mathcal{L}_\alpha)\right) = \text{CW}^{\tilde{\delta}}\left(\prod_{\alpha \in \Delta} D(\tilde{\delta}, \alpha)\right)_{\vert F_{\tilde{\delta}_1, \tilde{\delta}_2}}
		\end{align*}
		and therefore (\ref{FixedContr}) is equal to
		\begin{align*}
			\frac{1}{\vert W \vert}\sum_{\substack{\tilde{\delta}_1 \mapsto \delta_1\\\tilde{\delta}_2 \mapsto \delta_2}} \int_{[F_{\tilde{\delta}_1, \tilde{\delta}_2}]^\text{vir}}\frac{\text{CW}^{\tilde{\delta}}\left(P \prod_{\alpha \in \Delta} D(\tilde{\delta}, \alpha)\right)}{e^{\mathbb{C}^\ast}(N^\text{vir}_{F_{\tilde{\delta}_1, \tilde{\delta}_2}})}.
		\end{align*}
	\end{proof}
	This finally allows to prove the main result of this first part of the paper:
	\begin{theorem}[Abelianisation formula for quasimap invariants.]\label{MainTheoremAbFor}
		Consider a degree $\delta \in \chi(G)^\vee$ and a class $P \in A_G^\ast(\text{pt})$. The $\mathbb{C}^\ast$-equivariant quasimap invariant of $V/\!/G$ of degree $\delta$ with insertion $P$ can be computed as a sum of $\mathbb{C}^\ast$-equivariant quasimap invariants on $V/\!/T$ of degree $\tilde{\delta}\in \chi(T)^\vee$, where $\tilde{\delta}_{\vert \chi(G)} = \delta$:
		\begin{align*}
			\int_{[Q(V/\!/G, \delta)]^{\text{vir}}} \text{CW}^\delta(P) = \frac{1}{\vert W \vert} \sum_{\tilde{\delta}\mapsto \delta} \int_{[Q(V/\!/T, \tilde{\delta})]^{\text{vir}}} \text{CW}^{\tilde{\delta}}\left(P \prod_{\alpha \in \Delta} D(\tilde{\delta}, \alpha) \right).
		\end{align*} 
		Here $W$ is the Weyl group of $G$.
	\end{theorem}
	\begin{proof}
		By $\mathbb{C}^\ast$-localisation on $Q(V/\!/G, \delta)$ we can write
		\begin{align*}
			\int_{[Q(V/\!/G, \delta)]^{\text{vir}}} \text{CW}^\delta(P) &= \sum_{\delta_1+\delta_2 = \delta} \int_{[F_{\delta_1, \delta_2}]^{\text{vir}}} \frac{\text{CW}^\delta(P)}{e^{\mathbb{C}^\ast}(N^\text{vir}_{F_{\delta_1, \delta_2}})}\\	
			&= \frac{1}{\vert W \vert}\sum_{\delta_1+\delta_2 = \delta} \sum_{\substack{\tilde{\delta}_1 \mapsto \delta_1\\\tilde{\delta}_2 \mapsto \delta_2}} \int_{[F_{\tilde{\delta}_1, \tilde{\delta}_2}]^\text{vir}}\frac{\text{CW}^{\tilde{\delta}}\left(P \prod_{\alpha \in \Delta} D(\tilde{\delta}, \alpha)\right)}{e^{\mathbb{C}^\ast}(N^\text{vir}_{F_{\tilde{\delta}_1, \tilde{\delta}_2}})}.
		\end{align*} 
		By rearranging the sum we find that this quasimap invariant coincides with
		\begin{align*}
			&\frac{1}{\vert W \vert}\sum_{\tilde{\delta}\mapsto \delta}\sum_{\tilde{\delta}_1 + \tilde{\delta}_2 = \tilde{\delta}} \int_{[F_{\tilde{\delta}_1, \tilde{\delta}_2}]^\text{vir}}\frac{\text{CW}^{\tilde{\delta}}\left(P \prod_{\alpha \in \Delta} D(\tilde{\delta}, \alpha)\right)}{e^{\mathbb{C}^\ast}(N^\text{vir}_{F_{\tilde{\delta}_1, \tilde{\delta}_2}})}\\
			=& \frac{1}{\vert W \vert}\sum_{\tilde{\delta}\mapsto \delta} \int_{[Q(V/\!/T, \tilde{\delta})]^\text{vir}} \text{CW}^{\tilde{\delta}}\left(P \prod_{\alpha \in \Delta} D(\tilde{\delta}, \alpha)\right)
		\end{align*}
		again by $\mathbb{C}^\ast$-localisation on $Q(V/\!/T, \tilde{\delta})$.
	\end{proof}
	\begin{rem}
		The product $\prod_{\alpha \in \Delta} D(\tilde{\delta}, \alpha)$ can be simplified by noting that the roots of $G$ come in positive/negative pairs. In this way we find that
		\begin{align*}
			\prod_{\alpha \in \Delta} D(\tilde{\delta}, \alpha) = \prod_{\alpha \in \Delta^+}(-1)^{\langle\tilde{\delta},\alpha\rangle+1} \alpha \cdot (\alpha + \langle \tilde{\delta},\alpha\rangle z),
		\end{align*}
		which belongs to $A_G^\ast(\text{pt})[z]$.
	\end{rem}
	\begin{rem}
		Clearly the abelianisation formula works, by setting $z=0$ on the right-hand side of the formula, for non-equivariant invariants too.
	\end{rem}
	\section{Jeffrey-Kirwan formulae}\label{SectionJKFormulae}
	Let's quickly recall the content of the Jeffrey-Kirwan localisation formula \cite{JeffreyKirwan, SzenesVergne} specialised to our context.
	Let $\mathbb{C}^\ast$ act on $V$ via the cocharacter $\tilde{\delta}$. The $\mathbb{C}^\ast$\textit{-equivariant Kirwan map} for the action of $T$ on $V$ is the composition
	\begin{align*}
		r_{\tilde{\delta}} : A_{T\times \mathbb{C}^\ast}(V) \xrightarrow{i^\ast} A_{T\times \mathbb{C}^\ast}(V(T)^\text{ss}) \simeq A_{\mathbb{C}^\ast}(V/\!/T).
	\end{align*}
	In other words the Kirwan map is the ring homomorphism satisfying $r(z)=z$ and $r_{\tilde{\delta}}(\alpha) = c_1(\mathcal{L}_\alpha) - \langle \tilde{\delta}, \alpha \rangle z$ for all $\alpha \in \chi(T)$. The $\mathbb{C}^\ast$-equivariant version of the Jeffrey-Kirwan localisation theorem \cite{OntaniJKforDT} states that for a generic $z \in \mathbb{C}$ and for every $P \in A_{T \times \mathbb{C}^\ast}^\ast(V)$
	\begin{align*}
		\int_{V/\!/T}r_{\tilde{\delta}}(P) = \text{JK}_{-z\tilde{\delta}}^\mathfrak{A}\left(\frac{P}{\prod_{\rho \in \mathfrak{A}}(\rho+ \langle \tilde{\delta}, \rho \rangle z)}\right),
	\end{align*}
	where $\text{JK}$ denotes the Jeffrey-Kirwan residue as described in \cite{BrionVergne, SzenesVergne, OntaniJKforDT}.
	\subsection{The abelian case}
	By using this localisation formula we can express the equivariant quasimap invariants of $V/\!/T$ as residues, proving an equivariant version of the analogous result \cite[Equation (3.2)]{SzenesVergne} of Szenes and Vergne.
	\begin{pro}\label{JKAbelian}
		Let $\tilde{\delta} \in \chi(T)^\vee$ and $P \in A_T^\ast(\text{pt})$. Then, for a generic $z \in \mathbb{C}$, the equality 
		\begin{align*}
			\int_{[Q(V/\!/T, \tilde{\delta})]^{\text{vir}}} \text{CW}^{\tilde{\delta}}(P) = \sum_{\tilde{\delta}_1 + \tilde{\delta}_2 = \tilde{\delta}} \text{JK}^{\mathfrak{A}}_{-z\tilde{\delta}_2}\left(Z_{\tilde{\delta}}(-,z)P\right).
		\end{align*}
		holds true where the rational function $Z_{\tilde{\delta}} : \chi(T)^\vee_\mathbb{C}\times \mathbb{C} \dashrightarrow \mathbb{C}$ is defined as
		\begin{align*}
			Z_{\tilde{\delta}} := \prod_{\rho \in \mathfrak{A}}D(\tilde{\delta}, \rho)^{-1}.
		\end{align*}
		Here $Z_{\tilde{\delta}}(-,z)$ is the rational function on $\chi(T)_\mathbb{C}^\vee$ obtained by evaluating $Z_{\tilde{\delta}}$ at the chosen $z \in \mathbb{C}$.
	\end{pro}
	\begin{proof}
		This is an immediate application of the Jeffrey-Kirwan localisation formula. We start by $\mathbb{C}^\ast$-localisation on $Q(V/\!/T, \tilde{\delta})$, so the quasimap invariant is a sum of localised contributions:
		\begin{align*}
			\int_{[Q(V/\!/T, \tilde{\delta})]^{\text{vir}}} \text{CW}^{\tilde{\delta}}(P) = \sum_{\tilde{\delta}_1 + \tilde{\delta}_2 = \tilde{\delta}}\int_{[F_{\tilde{\delta}_1, \tilde{\delta}_2}]^\text{vir}}\frac{\text{CW}^{\tilde{\delta}}\left(P\right)}{e^{\mathbb{C}^\ast}(N^\text{vir}_{F_{\tilde{\delta}_1, \tilde{\delta}_2}})}.
		\end{align*}
		Fixed $\tilde{\delta}_1, \tilde{\delta}_2 \in \chi(T)^\vee$ summing to $\tilde{\delta}$, we can compute the contribution of the corresponding fixed locus $F_{\tilde{\delta}_1, \tilde{\delta}_2}$ by pushing forward through the evaluation map $\text{ev}: F_{\tilde{\delta}_1, \tilde{\delta}_2} \rightarrow V/\!/T$ first. The pushed forward class satisfies
		\begin{align*}
			\text{ev}_\ast \left(\text{CW}^{\tilde{\delta}}(P) \cap \frac{[F_{\tilde{\delta}_1, \tilde{\delta}_2}]^\text{vir}}{e^{\mathbb{C}^\ast}(N^\text{vir}_{F_{\tilde{\delta}_1, \tilde{\delta}_2}})}\right) &= r_{\tilde{\delta}_2}(P)\cap I_{\tilde{\delta}_1}^{V/\!/T}(z) I_{\tilde{\delta}_2}^{V/\!/T}(-z)\\
			&= r_{\tilde{\delta}_2}\left(P \prod_{\rho \in \mathfrak{A}}D(\tilde{\delta}, \rho)^{-1} (\rho+ \langle \tilde{\delta}_2, \rho \rangle z) \right) \! \cap \! [V/\!/T]
		\end{align*}
		by Lemma \ref{ChernWeilEvaluation}, Proposition \ref{EvaluatingVirtualCycle} and well known formulae for the abelian $I$-function (see for example \cite[Theorem 8.2]{WebbExpository}). This shows that
		\begin{align*}
			\int_{[F_{\tilde{\delta}_1, \tilde{\delta}_2}]^\text{vir}}\frac{\text{CW}^{\tilde{\delta}}\left(P\right)}{e^{\mathbb{C}^\ast}(N^\text{vir}_{F_{\tilde{\delta}_1, \tilde{\delta}_2}})} = \text{JK}^{\mathfrak{A}}_{-z\tilde{\delta}_2}\left(Z_{\tilde{\delta}}(-,z)P\right)
		\end{align*}
		by Jeffrey-Kirwan localisation.
	\end{proof}
	\begin{rem}
		In the proof of the previous result we pushed forward the computation to $V/\!/T$ and applied Jeffrey-Kirwan localisation there. We could have followed a different approach that we outline here. As seen in Example \ref{exampleToricqmaps1}, given $\tilde{\delta}\in \chi(T)^\vee$ the abelian moduli space $Q(V/\!/T,\tilde{\delta})$ is itself of the form $L(\tilde{\delta})/\!/T$, where $L(\tilde{\delta})$ is another linear space. The Chern-Weil homomorphism coincides, by definition, with the Kirwan map (in the sense of \cite{OntaniJKforDT}) for the action of $T$ on $L(\tilde{\delta})$, so the $\mathbb{C}^\ast$-equivariant quasimap invariants can be computed by Jeffrey-Kirwan localisation giving the same formula above. This was the approach taken by Szenes and Vergne in \cite{SzenesVergne} for computing nonequivariant quasimap invariants for toric targets.
	\end{rem}
	\subsection{The general case}
	By applying the abelianisation formula of Theorem \ref{MainTheoremAbFor}, we obtain the following expression of quasimap invariants of $V/\!/G$ as Jeffrey-Kirwan residues:
	\begin{theorem}\label{MainTheoremSingleInvariants}
		Let $\delta \in \chi(G)^\vee$ and $P \in A^\ast_G(\text{pt})$. Consider the linear space $\chi(T \times \mathbb{C}^\ast)_\mathbb{C}^\vee \simeq \chi(T)_\mathbb{C}^\vee \times \mathbb{C}$ and let $z$ be the coordinate on the second factor $\mathbb{C}$. Given $\tilde{\delta} \in \chi(T)^\vee$, consider the rational function
		\begin{align*}
			Z_{\tilde{\delta}} = \prod_{\rho \in \mathfrak{A}}D(\tilde{\delta}, \rho)^{-1}\prod_{\alpha \in \Delta}D(\tilde{\delta}, \alpha)
		\end{align*}
		on this linear space. The $\mathbb{C}^\ast$-equivariant quasimap invariant of degree $\delta$
		\begin{align*}
			\int_{[Q(V/\!/G, \delta)]^{\text{vir}}} \text{CW}^\delta(P) \in \mathbb{C}[z]_z
		\end{align*} 
		evaluated at a generic $z \in \mathbb{C}$ coincides with the sum of Jeffrey-Kirwan residues
		\begin{align*}
			\frac{1}{\vert W \vert}\sum_{\tilde{\delta}\mapsto \delta} \sum_{\tilde{\delta}_1+ \tilde{\delta}_2= \tilde{\delta}} \text{JK}^{\mathfrak{A}}_{-z\tilde{\delta}_2} \left(Z_{\tilde{\delta}}(-, z) P \right).
		\end{align*}
		Here $Z_{\tilde{\delta}}(-,z)$ is the rational function on $\chi(T)_\mathbb{C}^\vee$ obtained by evaluating $Z_{\tilde{\delta}}$ at the chosen $z \in \mathbb{C}$.
	\end{theorem}
	This settles the conjecture of Kim, Oh, Ueda and Yoshida \cite[Conjecture 10.10]{kimRMS} in the case of targets of the form $V/\!/G$. 
	\subsection{The non-equivariant limit}
	In the same way, as a corollary of the abelianisation formula we obtain an expression for the non-equivariant quasimap invariants of $V/\!/G$:
	\begin{cor}\label{nonequivariantInvariants}
		Let $\delta \in \chi(G)^\vee$ and $P \in A^\ast_G(\text{pt})$. The non-equivariant quasimap invariant of degree $\delta$
		\begin{align*}
			\int_{[Q(V/\!/G, \delta)]^{\text{vir}}} \text{CW}^\delta(P)
		\end{align*} 
		is the sum of Jeffrey-Kirwan residues
		\begin{align*}
			\frac{1}{\vert W \vert}\sum_{\tilde{\delta}\mapsto \delta} \text{JK}^\mathfrak{A}_{O} \left( \tilde{Z}_{\tilde{\delta}} P \right)
		\end{align*}
		where $\tilde{Z}_{\tilde{\delta}}$ is the rational function on $\chi(T)_\mathbb{C}^\vee$ defined by
		\begin{align*}
			\tilde{Z}_{\tilde{\delta}} = \prod_{\rho \in \mathfrak{A}}\rho^{-1-\langle\tilde{\delta},\rho\rangle} \prod_{\alpha \in \Delta^+} (-1)^{1+\langle \tilde{\delta}, \alpha\rangle} \alpha^2.
		\end{align*}
	\end{cor}
	\begin{proof}
		First we apply the abelianisation formula and we reduce to computing quasimap invariants of $V/\!/T$:
		\begin{align*}
			\int_{[Q(V/\!/G, \delta)]^{\text{vir}}} \text{CW}^\delta(P) = \frac{1}{\vert W \vert} \sum_{\tilde{\delta}\mapsto \delta} \int_{[Q(V/\!/T, \tilde{\delta})]^{\text{vir}}} \text{CW}^{\tilde{\delta}}\left(P \prod_{\alpha \in \Delta^+} (-1)^{1+\langle \tilde{\delta}, \alpha\rangle} \alpha^2 \right).
		\end{align*}
		In Example \ref{exampleToricqmaps1} we showed how the moduli space $Q(V/\!/T, \tilde{\delta})$ is isomorphic to the toric quotient $L(\tilde{\delta})/\!/T$ (in particular, it's irreducible). Its virtual fundamental class is the class
		\begin{align*}
			[Q(V/\!/T, \tilde{\delta})]^\text{vir} = \text{CW}^{\tilde{\delta}}\left(\prod_{\substack{\rho \in \mathfrak{A} \,\,\vert \,\, \langle\tilde{\delta}, \rho\rangle < 0}}\rho^{-1-\langle\tilde{\delta}, \rho\rangle}\right) \cap [Q(V/\!/T, \tilde{\delta})].
		\end{align*}
		and therefore the abelian quasimap invariant appearing on the right-hand side of the abelianisation formula can be computed as
		\begin{align*}
			\int_{Q(V/\!/T, \tilde{\delta})} \text{CW}^{\tilde{\delta}}\left( P \prod_{\alpha \in \Delta^+} (-1)^{1+\langle \tilde{\delta}, \alpha\rangle} \alpha^2 \prod_{\substack{\rho \in \mathfrak{A} \,\,\vert \,\, \langle\tilde{\delta}, \rho\rangle < 0}}\rho^{-1-\langle\tilde{\delta}, \rho\rangle} \right).
		\end{align*}
		The Chern-Weil homomorphism $\text{CW}^{\tilde{\delta}} : A^\ast_T(\text{pt}) \rightarrow A^\ast(Q(V/\!/T, \tilde{\delta}))$ coincides with the Kirwan map for the action of $T$ on $L(\tilde{\delta})$, so by Jeffrey-Kirwan localisation the integral above coincides with the claimed Jeffrey-Kirwan residue.
	\end{proof}
	\begin{rem}
		Residue formulae of this kind have been useful for matching, via direct computation, the two sides of mirror symmetry \cite{SzenesVergne, kimRMS, OntaniStoppa} in particular situations.
	\end{rem}
	\section{The Vafa-Intriligator formula}
	Consider a representation $V$ of a reductive group $G$, with maximal subtorus $T$, together with a linearisation corresponding to a character $\xi \in \chi(G)$. Assume that the $G$ and $T$-actions on the respective semistable loci are both free and that the quotients are proper. In this section we will mainly be interested in the formal sum
	\begin{align}\label{QuasimapSeries}
		\langle P \rangle^G(q) := \vert W \vert \sum_{\delta \in \chi(G)^\vee} q^\delta \int_{[Q(V/\!/G, \delta)]^{\text{vir}}} \text{CW}^{\delta}(P),
	\end{align}
	where $\vert W \vert$ denotes the order of the Weyl group of $G$.
	\begin{definition}
		Given a group $G$, we will denote with $\widecheck{G}$ the algebraic torus $\chi(G)_\mathbb{C}/\chi(G)$, which we call \textit{dual group of} $G$. 
	\end{definition}
	Given $q=[\psi] \in \widecheck{G}$ and $\delta \in \chi(G)^\vee$, setting
	\begin{align*}
		q^\delta := e^{2\pi i\langle \delta, \psi\rangle}
	\end{align*}
	we can think of the sum (\ref{QuasimapSeries}) as a function $\langle P \rangle^G : \widecheck{G} \dashrightarrow \mathbb{C}$ defined on the domain of convergence.
	\begin{definition}
		Given a basis $\lambda_1, \dots, \lambda_r$ of $\chi(G)^\vee$, the corresponding functions $q_1:= q^{\lambda_1}, \dots, q_r:= q^{\lambda_r}$ on $\widecheck{G}\simeq (\mathbb{C}^\ast)^r$ defined as above by
		\begin{align}\label{coordinatesOnDualGroup}
			q_i : \widecheck{G} \rightarrow \mathbb{C}^\ast \quad : \quad q_i[\psi]:= e^{2\pi i\langle \lambda_i, \psi\rangle}.  
		\end{align}
		define coordinates on $\widecheck{G}$. We call them \textit{coordinates associated to} $\lambda_1, \dots, \lambda_r$.
	\end{definition}
	Here we address the following questions:
	\begin{enumerate}
		\item\label{question1} Does $\langle P \rangle^G$ converge at any point of $\widecheck{G}$?
		\item\label{question2} If yes, can we find a closed expression for the function it converges to?
	\end{enumerate}
	\subsection{Abelianisation and convergence}
	Let $\Delta\subset \chi(T)$ be the set of roots of $G$, $\Delta^+ \subset \Delta$ a choice of positive roots and consider the involution
	\begin{align}\label{sigmaInvolution}
		\sigma : \widecheck{T} \rightarrow \widecheck{T} \quad : \quad \sigma(q) := q+ \left[\frac{1}{2}\sum_{\alpha \in \Delta^+}\alpha\right].
	\end{align}
	\begin{rem}
		This involution doesn't depend on the choice of the positive roots. Indeed, assume we replace $\alpha \in \Delta^+$ with its opposite $-\alpha$. Then $\sigma$ changes by
		\begin{align*}
			\sigma_{\text{new}}(q) = \sigma_{\text{old}}(q) - \alpha
		\end{align*}
		and these two elements are the same in $\widecheck{T}:= \chi(T)_\mathbb{C}/\chi(T)$.
	\end{rem}
	\begin{rem}
		The rational character of $T$ given by the half-sum of the positive roots $\frac{1}{2}\sum_{\alpha \in \Delta^+}\alpha$ plays an important role in representation theory and it is called the \textit{Weyl vector} of $G$. This is not only a character of the maximal subtorus $T$ but, being invariant under the action of the Weyl group $W$ of $G$, it extends to a character of $G$. In particular, this shows that $\widecheck{G}$ is an invariant subtorus of the involution $\sigma$.
	\end{rem}
	Defined the function $\langle P \rangle^T : \widecheck{T} \dashrightarrow \mathbb{C}$ as
	\begin{align}\label{QuasimapSeriesAb}
		\langle P \rangle^T(q) := \sum_{\lambda \in \chi(T)^\vee} q^\lambda \int_{[Q(V/\!/T, \lambda)]^{\text{vir}}} \text{CW}^{\lambda}\big(P \prod_{\alpha \in \Delta}\alpha\big),
	\end{align}
	by the abelianisation formula we have the following
	\begin{theorem}\label{AbelianisationSeries}
		Consider the inclusion $\widecheck{G} \hookrightarrow \widecheck{T}$. The equality
		\begin{align*}
			\langle P \rangle^G(q) = \langle P \rangle^T(\sigma(q)).
		\end{align*}
		holds true for every $q \in \widecheck{G}$ so that one of the two sums converges.
	\end{theorem}
	\begin{proof}
		This is an immediate application of the abelianisation formula of Theorem \ref{MainTheoremAbFor} which allows to write (\ref{QuasimapSeries}) in terms of invariants of $V/\!/T$.
	\end{proof}
	In order to show that $\langle P \rangle^G$  converges somewhere on $\widecheck{G}$, thanks to the result above we can study the intersection of the convergence set of the formal sum $\langle P \rangle^T$ on $\widecheck{T}$ with the subtorus $\widecheck{G}\subseteq \widecheck{T}$. The semi-positivity of the triple $(T,V,\xi)$, in the sense of \cite[Section 1.4]{KimCiocanWallCrossing}, is sufficient for having convergence results.
	\begin{definition}
		Let $\text{NE}(\xi) \subseteq \chi(T)_\mathbb{R}^\vee$ be the cone generated by degrees of stable quasimaps (from curves of any genus). The triple $(V,T,\xi)$ is called \textit{semi-positive} if
		\begin{align*}
			\langle \tilde{\delta}, \kappa \rangle \geq 0 \qquad \forall \tilde{\delta} \in \text{NE}(\xi)
		\end{align*}
		where $\kappa := \sum_{\rho \in \mathfrak{A}} \rho \in \chi(T)$ is the \textit{anticanonical character}.
	\end{definition}
	\begin{rem}
		In particular, if $\tilde{\delta}$ is not a $\xi$-effective class, then the moduli space $Q(V/\!/T,\tilde{\delta})$ is nonempty and $\tilde{\delta}$ doesn't contribute to the power series $\langle P \rangle^T$.
	\end{rem}
	Let $\mathfrak{c} \subset \chi(T)_\mathbb{R}$ the be the open chamber of the momentum cone for the action $T \curvearrowright V$ which the stability $\xi$ belongs to.
	By \cite[Theorem 1.7]{Shoemaker} we see that $\text{NE}(\xi) = \mathfrak{c}^\vee$ and therefore
	\begin{lem}
		The triple $(T,V,\xi)$ is \textit{semi-positive} if $\kappa \in \overline{\mathfrak{c}}$, or in other words
		\begin{align*}
			\langle \tilde{\delta}, \kappa \rangle \geq 0 \qquad \forall \tilde{\delta} \in \mathfrak{c}^\vee.
		\end{align*}
	\end{lem}
	\subsubsection{The positive case}
	Assume that $(V,T,\xi)$ is \textit{positive}, namely $\kappa \in \mathfrak{c}$. This means that for every $\delta \in \mathfrak{c}^\vee$ the anticanonical character pairs positively with it: $\langle \tilde{\delta}, \kappa \rangle >0$. We will need the following elementary lemma:
	\begin{lem}\label{positive_cone_lemma}
		Let $\kappa \in \mathfrak{c}$ and let $d>0$. There is only a finite number of $\tilde{\delta} \in \mathfrak{c}^\vee \cap \chi(T)^\vee$ such that $\langle \kappa, \delta \rangle\leq d$.
	\end{lem} 
	\begin{proof}
		Since $\overline{\mathfrak{c}}$ is rational polyhedral then $\mathfrak{c}^\vee$ is rational polyhedral, hence it admits a finite number of generators $\tilde{\delta}_1, \dots, \tilde{\delta}_l$. For each of them we set $m_i:=\langle \tilde{\delta}_i, \kappa \rangle > 0$ (since $\kappa$ is in the open cone $\mathfrak{c}$) and therefore the general element of $\mathfrak{c}^\vee$ 
		\begin{align*}
			\tilde{\delta}:= \sum_{j=1}^l \alpha_j \tilde{\delta}_j
		\end{align*}
		pairs with $\kappa$ as
		\begin{align*}
			\langle\tilde{\delta}, \kappa \rangle:= \sum_{j=1}^l \alpha_j m_j,
		\end{align*}
		with $\alpha_j >0$. In order for this pairing to be less then or equal to $d$ we need $\alpha_j \leq \frac{d_j}{m_j}$ for every $j$, hence the set of elements $\tilde{\delta} \in \mathfrak{c}^\vee$ pairing with $\kappa$ to something less or equal to $d$ is bounded. In particular, intersecting it with the lattice $\chi(T)^\vee$ gives only a finite number of points.
	\end{proof}
	This implies the following convergence result:
	\begin{pro}
		Let $(V,T,\xi)$ be positive, or equivalently $\kappa \in \mathfrak{c}$. Then, for every $P \in A_T^\ast(\text{pt})$, the power series $\langle P \rangle^T$ and $\langle P \rangle^G$ have only finitely many nonzero terms.
	\end{pro}
	\begin{proof}
		We start by noting that
		\begin{align*}
			\text{virtual dim}(Q(V/\!/T, \tilde{\delta})) = \text{dim}(V) - \text{dim}(T) + \langle \tilde{\delta}, \kappa \rangle
		\end{align*}
		by example (\ref{exampleToricqmaps1}). Let $d$ be the maximal degree appearing in $P$. From the previous Lemma \ref{positive_cone_lemma} we see that that there is only a finite number of $\xi$-effective classes $\tilde{\delta}$ so that 
		\begin{align*}
			\text{virtual dim}(Q(V/\!/T, \tilde{\delta})) \leq d
		\end{align*}
		and therefore only a finite number of degrees $\tilde{\delta}$ contribute to the sum. The analogous statement for $\langle P \rangle^G$ follows from this and the abelianisation formula.
	\end{proof}
	\subsubsection{The semi-positive case}
	We have seen that in the positive case the function $\langle P \rangle^G$ is regular. In the strictly semi-positive case, namely when $\kappa \in \partial \overline{\mathfrak{c}}$, even though the series $\langle P \rangle^G$ contains infinitely many terms, we can use a result of Szenes and Vergne to argue that it converges somewhere.  
	
	In \cite[Definition 3.3 and discussion below]{SzenesVergne}, Szenes and Vergne show that only a simplicial cone in $\chi(T)^\vee$ contributes to the sum (i.e. for all $\lambda$ not in the cone, the coefficient of $q^\lambda$ is zero), hence $\langle P \rangle^T$ is a honest power series in those coordinates. Moreover, they show that it has nonzero radius of convergence \cite[Lemma 3.3]{SzenesVergne}:
	\begin{theorem}[Szenes-Vergne]\label{SzenesVergneTheorem1}
		For every integral basis $\lambda^1, \dots, \lambda^r$ of $\chi(T)$ contained in $\overline{\mathfrak{c}}$, the dual basis $\boldsymbol{\lambda}$ is called a $\mathfrak{c}$-positive basis. Let $\boldsymbol{\lambda}$ be a $\mathfrak{c}$-positive basis. If $\lambda \in \chi(T)^\vee$ is a degree whose corresponding quasimap space $Q(V/\!/T,\lambda)$ has nonzero virtual fundamental class $[Q(V/\!/T,\lambda)]^{\text{vir}}\neq 0$, then $\lambda$ is contained in the simplicial cone spanned by $\boldsymbol{\lambda}$. This makes $\langle P \rangle^T$ a power series in the coordinates defined by $\boldsymbol{\lambda}$, converging on a nonempty open polydisk centered at the origin.
	\end{theorem}
	\begin{rem}
		The choice of a $\mathfrak{c}$-positive basis $\boldsymbol{\lambda}$ identifies $\widecheck{T}$ with $(\mathbb{C}^\ast)^r$ and in particular it identifies the points of $\widecheck{T}$ that are close to the origin of $\mathbb{C}^r$. We think of this origin is a \textit{limit point} of $\widecheck{T}$, which we will denote with $O$. The series $\langle P\rangle^T$ converges in a neighborhood of this limit point. Different choices of such bases $\boldsymbol{\lambda}$ correspond to different limit points in $\widecheck{T}$, and $\langle P\rangle^T$ converges around all of them.
	\end{rem}
	\begin{ex}
		Here we try to clarify the situation with some pictures. Assume that the momentum cone has two chambers $\mathfrak{c}_1, \mathfrak{c}_2$ and that our stability $\xi$ belongs to $\mathfrak{c}:= \mathfrak{c}_1$. Assume moreover that $\mathfrak{c}$ can be split into two simplicial cones $\mathfrak{s}_1, \mathfrak{s}_2$ spanned by integral bases.
		\[\includegraphics[width=.7\textwidth, keepaspectratio]{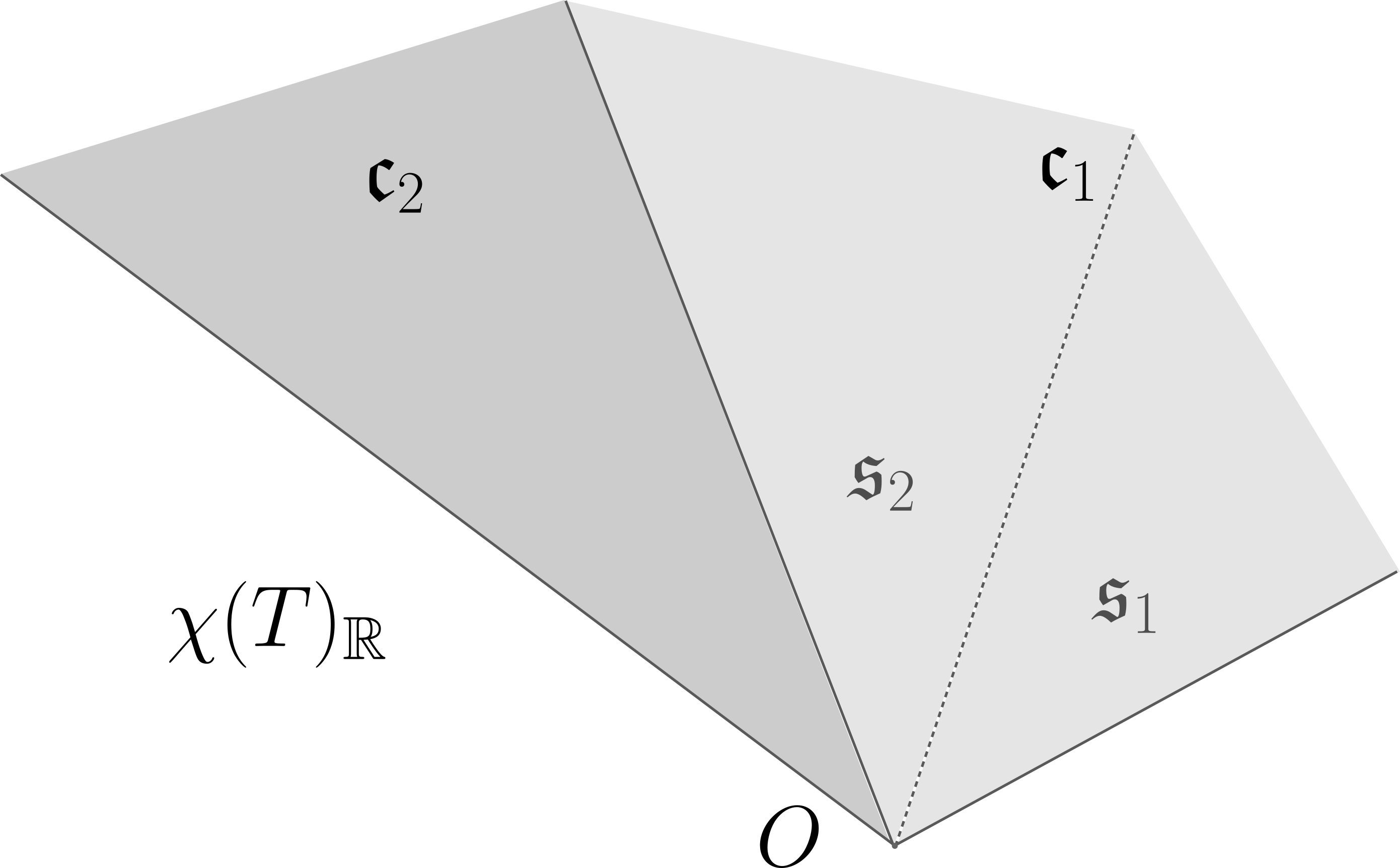}\]
		Each of these two bases determines a set of coordinates on $\widecheck{T}$, and the formal sum $\langle P \rangle^T$ becomes a power series with nonzero convergence radius in such coordinates. In particular, it converges in an analytic neighborhood $N_{\mathfrak{s}_1}$ and $N_{\mathfrak{s}_2}$ of the origin in such coordinates.
		\[\includegraphics[width=.5\textwidth, keepaspectratio]{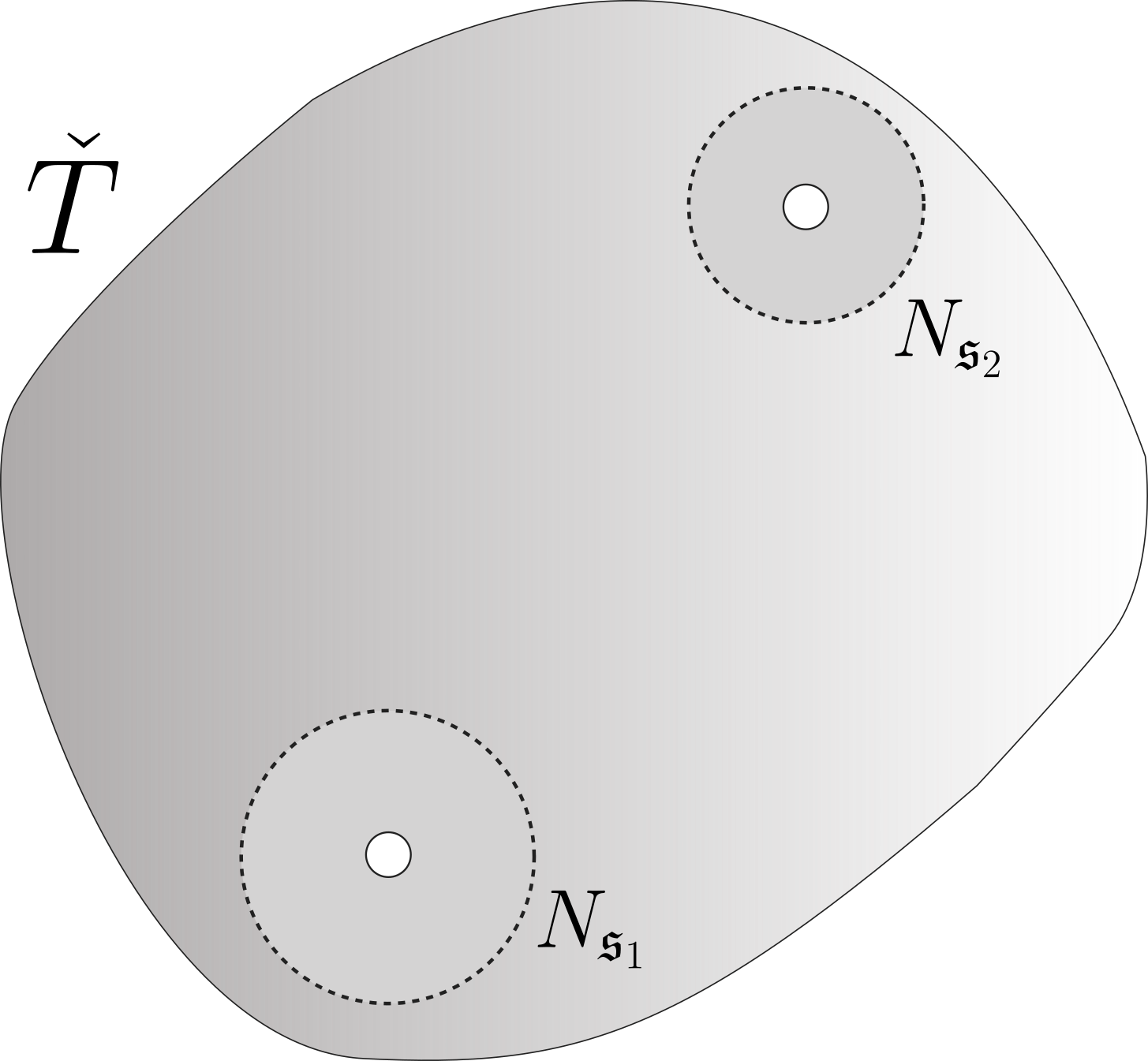}\]
	\end{ex}
	We now want to show that $\langle P \rangle^G$ converges somewhere on $\widecheck{G}$, and we want to deduce this from what we know about the domain of convergence of $\langle P \rangle^T$.
	\begin{lem}\label{inversePolydiskLemma}
		Consider the morphism $\chi(T)_\mathbb{R} \rightarrow \widecheck{T}$ given by $\psi \mapsto [i \psi]$ and the inverse image of the convergence domain of $\langle P \rangle^T$
		\begin{align*}
			D:= \left\lbrace \psi \in \chi(T)_\mathbb{R} \quad \vert \quad \langle P \rangle^T \text{ converges absolutely at } [i \psi] \right\rbrace.
		\end{align*} For every $\mathfrak{c}$-positive basis $\boldsymbol{\lambda}$ of $\chi(T)^\vee$ there are positive numbers $\epsilon_1, \dots, \epsilon_r>0$ such that the set
		\begin{align*}
			\left\lbrace \sum_{i=1}^r a_i \lambda_i \quad \vert \quad a_i \geq \epsilon_i \,\, \forall \,\, i \right\rbrace \subset \chi(T)^\vee_\mathbb{R}
		\end{align*}
		is contained in $D$.
	\end{lem}
	\begin{proof}
		A $\mathfrak{c}$-positive basis $\boldsymbol{\lambda}$ of $\chi(T)^\vee$ defines coordinates on $\widecheck{T}$ as in (\ref{coordinatesOnDualGroup}), and the morphism above becomes
		\begin{align*}
			\mathbb{R}^r \rightarrow (\mathbb{C}^\ast)^r \quad : \quad (x_1, \dots, x_r) \mapsto (e^{-2\pi x_1}, \dots, e^{-2\pi x_r}).
		\end{align*}
		By Theorem \ref{SzenesVergneTheorem1}, $\langle P \rangle^T$ is a power series on $(\mathbb{C}^\ast)^r$ and there are positive numbers $\nu_1, \dots, \nu_r>0$ such that this series converges on the intersection of the polydisk $\Delta(O, \nu_1, \dots, \nu_r)$ with $(\mathbb{C}^\ast)^r$. The thesis follows by considering the inverse image of this intersection.
	\end{proof}
	\begin{figure}[h]
		\centering
		\includegraphics[width=.5\textwidth, keepaspectratio]{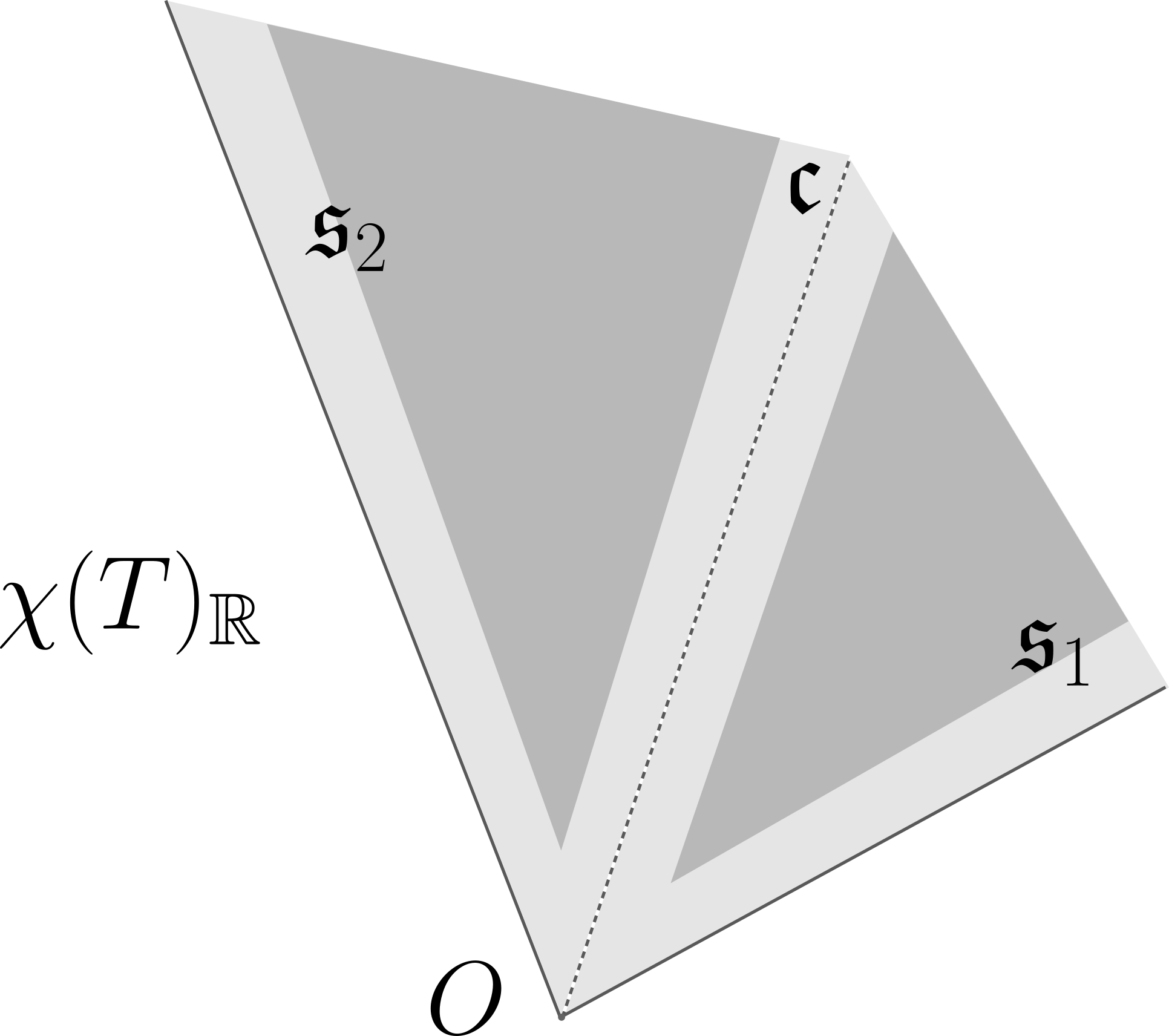}
		\caption{For every simplicial cone $\mathfrak{s}$ spanned by integral elements of $\mathfrak{c}$, a translation $\mathfrak{s}$ is contained in the inverse image $D$ of the convergence domain of $\langle P \rangle^T$. Being $\chi(G)_\mathbb{R}$ a linear subspace, this means that if it intersects the interior of $\mathfrak{s}$, then it intersects $D$ too.}
		\label{figurePolytope}
	\end{figure}
	\begin{lem}\label{intersectionLemma}
		Let $L$ be a ray centered at the origin of $\chi(T)_\mathbb{R}$ and contained in the interior of the chamber $\mathfrak{c}$. Then $D \cap L \neq \emptyset$.
	\end{lem}
	\begin{proof}
		Consider a decomposition of the chamber $\mathfrak{c}$ into simplicial cones. The ray $L$ is contained in the relative interior of a codimension $k$ cone appearing as a wall in this decomposition. The proof is by induction on $k$. The base case is $k=0$, so $L$ is contained in the interior of one of this cones and the thesis follows immediately from Lemma \ref{inversePolydiskLemma}. We now focus on the inductive step. Let $\sigma$ be a simplicial cone appearing in the decomposition and having a codimension $k$ wall $w$ containing $L$ in its relative interior. By inductive hypothesis $D \cap \sigma$ is a nonempty open subset invariant under multiplication by $\lambda > 1$. Then consider the space of combinations
		\begin{align*}
			U := \left\lbrace (1-t)p + t l \quad \vert \quad t>1, \, p \in D \cap \sigma, \, l \in L \right\rbrace,
		\end{align*}		
		which again is invariant under multiplication by $\lambda>1$. Since $L$ is in the interior of $\mathfrak{c}$, this open set $U$ must intersect some other top dimensional simplicial cone $\nu$ appearing in the decomposition of $\mathfrak{c}$ in some interior point $Q$. Then the ray spanned by $Q$ intersects $D$ and therefore $U$ intersects $D$ too. Since $D$ is convex by the general theory of power series (convergence domains are logarithmically convex), this means that $L$ intersects $D$. 
	\end{proof}
	We are finally ready to provide an answer to question \ref{question1}:
	\begin{theorem}\label{theoremQuestion1}
		If the triple $(V,T,\xi)$ is semi-positive, then the series $\langle P \rangle^G$ converges at some point of $\widecheck{G}$.
	\end{theorem}
	\begin{proof}
		Let $L$ be the ray contained in $\chi(G)_\mathbb{R} \cap \mathfrak{c}$ spanned by the stability $\xi$. By Lemma \ref{intersectionLemma} we have that $L \cap D \neq \emptyset$. If $\psi$ is a point of $L \cap D$, then $\langle P \rangle^T$ converges at $[i\psi] \in \widecheck{G}$ and the same does $\langle P \rangle^G$ by the abelianisation theorem \ref{AbelianisationSeries}.
	\end{proof}
	\subsection{The formula}\label{subsectionVafaIntriligatorFormula}
	Here we address question \ref{question2}, giving a closed formula for $\langle P \rangle^G$. Let $U(\mathfrak{A})\subset \chi(T)_\mathbb{C}^\vee$ be the open subset where all the weights $\rho \in \mathfrak{A}$ don't vanish. We start by considering the morphism
	\begin{align}\label{pMorphism}
		p : U(\mathfrak{A}) \rightarrow \widecheck{T} \quad : \quad p(u):= \sum_{\rho \in \mathfrak{A}} \frac{\log\left(\rho(u)\right)}{2 \pi i}[\rho].
	\end{align}
	Let $\boldsymbol{\lambda} = \lbrace \lambda_1, \dots, \lambda_r \rbrace$ be an integral basis of $\chi(T)^\vee$ and let $\lbrace \lambda^1, \dots, \lambda^r \rbrace$ be the dual basis. 
	They determine an open embedding $U(\mathfrak{A}) \hookrightarrow \mathbb{C}^r$, an isomorphism $\widecheck{T}\simeq (\mathbb{C}^\ast)^r$ and a top form $d\mu := d\lambda^1 \wedge \dots \wedge d\lambda^r$ on $U(\mathfrak{A})$. In these coordinates the morphism $p$ becomes
	\begin{align*}
		p : U(\mathfrak{A}) \rightarrow (\mathbb{C}^\ast)^r \quad : \quad p_i(u) = \prod_{\rho\in \mathfrak{A}} \rho(u)^{\langle \lambda_i, \rho\rangle},
	\end{align*}
	which is related to the regular function
	\begin{align*}
		D_\mathfrak{A} : U(\mathfrak{A}) \rightarrow \mathbb{C} \quad : \quad D_\mathfrak{A}(u):= \text{det}\left(\sum_{\rho \in \mathfrak{A}} \frac{\langle \lambda_i, \rho\rangle\langle \lambda_j, \rho\rangle}{\rho(u)}\right)_{i,j=1, \dots, r}
	\end{align*}
	as shown by the following result \cite[Proposition 4.5 (3)]{SzenesVergne}:
	\begin{lem}\label{JacobianP}
		For every $u \in U(\mathfrak{A})$ the following forms
		\begin{align*}
			\frac{dp_1}{p_1(u)} \wedge \dots \wedge \frac{dp_r}{p_r(u)} = D_\mathfrak{A}(u) d\mu
		\end{align*}
		coincide. In particular, $D_\mathfrak{A}$ doesn't depend on the choice of the integral basis $\boldsymbol{\lambda}$, and the function $p$ has nonzero Jacobian at every point where the function $D_\mathfrak{A}$ is nonzero. 
	\end{lem}
	These facts allow us to prove the following fact appearing in \cite[Corollary 4.6]{SzenesVergne}:
	\begin{lem}\label{finiteOrbit}
		The morphism $p : U(\mathfrak{A}) \rightarrow \widecheck{T}$ is dominant and there is an open subscheme $V \subseteq \widecheck{T}$ so that the restriction
		\begin{align*}
			p : p^{-1}(V) \rightarrow V
		\end{align*}
		is \'etale. In particular, for all points $q$ in that open subset, the fiber $p^{-1}(q)$ is the disjoint union of a finite number of reduced points. 
	\end{lem}
	\begin{proof}
		The function $D_\mathfrak{A}$ is nonzero: this can be seen from \cite[Proposition 4.5, (1)]{SzenesVergne}, using the fact that $\mathfrak{A}$ is projective as in the proof of \cite[Corollary 4.6]{SzenesVergne}. In particular, since there is a point where the Jacobian of $p$ is nonzero, the image of $p$ is of dimension $r$, hence $p$ is dominant. Moreover the vanishing locus of $D_\mathfrak{A}$ is of dimension strictly smaller than $r$, so its image in $\widecheck{T}$ through $p$ is positive codimensional. Its complement contains an nonempty open subscheme $V$, and its elements $q$ are such that $p$ is \'etale at all elements of the fiber $p^{-1}(q)$. 
	\end{proof}
	We can finally use the abelianisation formula to answer question \ref{question2}:
	\begin{definition}
		Given $q \in \widecheck{T}$ we consider the sum
		\begin{align}\label{sum_VI}
			\sum_{w \in p^{-1}(q)} \frac{P(w) \prod_{\alpha \in \Delta}\alpha(w)}{D_{\mathfrak{A}}(w)\prod_{\rho \in \mathfrak{A}}\rho(w)},
		\end{align}
		which is generically well defined by the previous Lemma \ref{finiteOrbit}.
	\end{definition}
	\begin{theorem}[Vafa-Intriligator formula]\label{theoremQuestion2}
		Consider a finite dimensional representation $V$ of a reductive group $G$ together with a GIT stability $\xi \in \chi(G)$ so that the action on the semistable locus is free and $V/\!/T$ is proper and semi-positive, where $T \subseteq G$ is a maximal subtorus. Given a class $P \in A^\ast_G(\text{pt})$, the sum (\ref{sum_VI}) extends to a rational function
		\begin{align*}
			\langle P\rangle_{\mathfrak{B}} : \widecheck{T} \dashrightarrow \mathbb{C}.
		\end{align*}
		Let $q \in \widecheck{G}$ belong to the domain of convergence of the power series $\langle P \rangle^G$. Then it belongs to the domain of definition of $\langle P \rangle_{\mathfrak{B}}$ and 
		\begin{align*}
			\langle P \rangle^G(q) = \langle P \rangle_\mathfrak{B}(\sigma(q)).
		\end{align*}
	\end{theorem}
	\begin{proof}
		The equality $\langle P \rangle^T(q) = \langle P \rangle_\mathfrak{B}(q)$ for every $q \in \widecheck{T}$ so that the series $\langle P \rangle^T$ converges is the main result of the paper of Szenes and Vergne \cite[Theorem 4.1]{SzenesVergne}. Its complete formulation is given in Theorem \ref{SzenesVergneVers1} in the Appendix. The claim follows from this equality combined with Theorem \ref{theoremQuestion1}.
	\end{proof}
	\begin{rem}
		In principle it might happen that, for every point $q \in \widecheck{G}$ where $\langle P \rangle^G(q)$ converges but the sum (\ref{sum_VI}) in not well defined, for example because $D_\mathfrak{A}$ vanishes on $p^{-1}(q)$. This result claims that even if the sum is not defined, the induced rational function $\langle P \rangle_{\mathfrak{B}}$ does extend to $q$, and $\langle P \rangle^G(q) = \langle P \rangle_{\mathfrak{B}}(q)$.
		
		On the other hand, if there is a point $q \in \widecheck{G}$ so that the sum (\ref{sum_VI}) is well defined, then it is defined for a generic $q \in U \subseteq \widecheck{G}$ and it coincides with $\langle P \rangle^G$ on the intersection of $U$ with the convergence domain of  $\langle P \rangle^G$.
	\end{rem}
	\begin{rem}
		Notice that the left-hand side $\langle P \rangle^G(q)$ depends on the stability parameter $\xi \in \chi(G)$, while the right-hand side $\langle P \rangle_{\mathfrak{B}}(\sigma(q))$ does not. The parameter $\xi$ plays the role of selecting the locus of $\widecheck{G}$ where the power series on the left converges.
	\end{rem}
	\begin{ex}[Intersection numbers on Quot schemes]
		It's easy to see that this formula reduces to the formula of \cite[Theorem 3]{MarianOpreaVI} by Marian and Oprea in the case where the target $V/\!/G$ is the Grassmannian $\text{Gr}_r(\mathbb{C}^n)$. This corresponds to the choice $V= \text{Mat}_{r \times n}(\mathbb{C})$ and $G:= \text{GL}_r(\mathbb{C})$ acting by multiplication on the left. If $T=(\mathbb{C}^\ast)^r$ is the diagonal torus in $G$, with the canonical identifications $\chi(T)^\vee_\mathbb{C}\simeq \mathbb{C}^r$ and let $u_1, \dots, u_r$ be the coordinates there. Then it's easy to check that
		\begin{align*}
			p(u) &= (u_1^n, \dots, u_r^n) \text{ and }
			D_\mathfrak{A}(u) = n^r \prod_{i=1}^r u_i^{-1},
		\end{align*}
		and that $\sigma(q) = (-1)^{r-1}q$,	so that the formula becomes
		\begin{align*}
			\langle P \rangle^G(q) = \frac{1}{n^r} \sum_{u_1^n = (-1)^{r-1}q} \cdots \sum_{u_r^n = (-1)^{r-1}q} \frac{P(u_1, \dots, u_r) \prod_{i\neq j}(u_i-u_j)}{\prod_{k=1}^r u^{n-1}_k}.
		\end{align*}
		This becomes the formula of Marian and Oprea once we realise that in this case the function that associates to $\delta \in \chi(G)^\vee \simeq \mathbb{Z}$ the dimension $e(\delta)$ of the quasimap space $Q(\text{Gr}_r(\mathbb{C}^n), \delta) \simeq \text{Quot}_{\mathbb{P}^1}(\mathcal{O}^{\oplus n}, r, \delta)$ is injective:
		\begin{align*}
			e(\delta) = n \delta + r(n-r).
		\end{align*}
		Indeed, given a symmetric polynomial $P \in \mathbb{C}[u_1, \dots, u_r]^{\mathfrak{S}_r}$ which is homogeneous of degree $e(\delta)$, we find that
		\begin{align*}
			\langle P \rangle(q) = r! \, q^{\delta} \int_{\text{Quot}_{\mathbb{P}^1}(\mathcal{O}^{\oplus n}, r, \delta)} \text{CW}^\delta(P).
		\end{align*}
		By evaluating our formula at $q = (-1)^{r-1}$ we recover the result of Marian and Oprea:
		\begin{align*}
			\int_{\text{Quot}_{\mathbb{P}^1}(\mathcal{O}^{\oplus n}, r, \delta)} \!\!\text{CW}^\delta(P)
            =& \frac{(-1)^{d(r-1)}}{r! \, n^r} \sum_{u_1^n = 1} \cdots \sum_{u_r^n = 1} \frac{P(u_1, \dots, u_r) \prod_{i\neq j}(u_i-u_j)}{\prod_{k=1}^r u^{n-1}_k}\\
			=& \frac{(-1)^{\epsilon}}{r! \, n^r} \sum_{u_1^n = 1} \cdots \sum_{u_r^n = 1} P(u_1, \dots, u_r) \prod_{i< j}(u_i-u_j)^2\prod_{k=1}^r u_k.
		\end{align*}
		where $\epsilon = d(r-1) + \binom{r}{2}$. Notice that the factor $\frac{1}{r!}$ doesn't appear in Marian and Oprea's formula, since their sum is over \textit{unordered} $r$-tuples of roots of unity.
	\end{ex}
	\appendix
	\section{Proof of Lemma \ref{generalisation_webb}}\label{appendix_proof_key_proposition}
	In this section we want to adapt the proof of \cite[Theorem 1.1.1]{WebbAbelianNonAbelian} to prove the slight generalisation we need, namely Lemma \ref{generalisation_webb}. No argument here is really original: we just check that Webb's proof goes through without obstructions.
	
	We need to introduce some further notation in order to make the formulae and computations easier to read:
	\begin{itemize}
		\item consider a ring $R$ and the corresponding ring $R[z]_z$ of Laurent polynomials. Given $P \in R[z]_z$, we will denote with $\overline{P} \in R[z]_z$ the Laurent polynomial defined as $\overline{P}(z):= P(-z)$.
		\item Whenever we look at a $\mathbb{C}^\ast$-fixed locus $F$ in a space $X$ with an equivariant obstruction theory, we will use the notation
		\begin{align*}
			[F]^{\text{vir}}_\text{loc} := \frac{[F]^{\text{vir}}}{e^{\mathbb{C}^\ast}(N^\text{vir}_{F/X})}
		\end{align*}
		for the localised virtual fundamental class of $F$.
		\item Finally, whenever we have a Chow cohomology class $\alpha$ defined on $X$, we will write $\alpha \cap [F]^{\text{vir}}_\text{loc}$ in order to denote the pairing $\alpha_{\vert F} \cap [F]^{\text{vir}}_\text{loc}$ in the equivariant Chow group of $F$, omitting the restriction symbol.
	\end{itemize}
	\subsection{Abelianisation and fixed loci}
	We start by recalling some important ingredients from Webb's work. 
	\begin{rem}\label{action_inversion}
		In \cite{WebbAbelianNonAbelian} all results are for the component $F_{\delta,0}$ of the fixed locus, but they immediately imply those for the component $F_{0,\delta}$, that we are interested in, since the two components are mapped one into the other by the automorphism of $Q(V/\!/G,\delta)$ induced by the automorphism inverting the coordinates on $\mathbb{P}^1$. This automorphism inverts the action of $\mathbb{C}^\ast$ on $\mathbb{P}^1$, hence on the moduli space, thus all results for $F_{0,\delta}$ are the same as the results for $F_{\delta, 0}$ with the opposite action.
	\end{rem} 
	Let $\tilde{\delta} \mapsto \delta$ and consider the inverse image of the open subscheme $V(G)^\text{ss}/\!/T$ through the evaluation morphism:
	\begin{equation}\label{openDiagram}
		\begin{tikzcd}
			F^0_{0,\tilde{\delta}} \arrow[rd, phantom, "\square"] \arrow[r, hook, "h"] \arrow[d, "\text{ev}"] & F_{0,\tilde{\delta}} \arrow[d, "\text{ev}"]\\
			V(G)^\text{ss}/\!/T \arrow[r, hook, "j"] & V/\!/T
		\end{tikzcd}.
	\end{equation}
	On this locus $F^0_{0,\tilde{\delta}}$ there is a well defined morphism to $F_{0,\delta}$. Indeed, given a $\mathbb{C}^\ast$-fixed quasimap to $V/\!/T$ mapping to $V(G)^\text{ss}/\!/T$ via the evaluation map, we can build a $G$-principal bundle mapping equivariantly to $V$ by
	\[
	\begin{tikzcd}
		P \arrow[r]\arrow[d] & V\\
		\mathbb{P}^1 & 
	\end{tikzcd} \mapsto
	\begin{tikzcd}
		P \times_T G \arrow[r]\arrow[d] & V\\
		\mathbb{P}^1 & 
	\end{tikzcd}
	\]
	and this defines a quasimap to $V/\!/G$. This construction immediately generalises to families, again by extending the structure group of the principal bundle from $T$ to $G$, and defines a morphism $\psi : F^0_{0,\tilde{\delta}} \rightarrow F_{0,\delta}$, whose image we will denote with $k : \mathcal{F}_{0,\tilde{\delta}} \hookrightarrow F_{0,\delta}$.
	If we denote with $P_{\tilde{\delta}}$ the parabolic subgroup of $G$ defined by $\tilde{\delta}$ \cite[Equation (35)]{WebbAbelianNonAbelian} we obtain the following diagram
	\begin{equation}\label{pfDiagram}
		\begin{tikzcd}
			F^0_{0,\tilde{\delta}} \arrow[r, swap, "\psi"] \arrow[d, "\text{ev}"]\arrow[rd, phantom, "\square"] & \mathcal{F}_{0,\tilde{\delta}} \arrow[d, "i"] \arrow[r, swap, hook, "k"] & {F}_{0,\delta} \arrow[d, swap, "\text{ev}"]\\
			V(G)^\text{ss}/\!/T \arrow[r, "p"] & V(G)^\text{ss}/P_{\tilde{\delta}} \arrow[r, "f"] & V/\!/G 
		\end{tikzcd}
	\end{equation}
	which coincides with part of \cite[Diagram (33)]{WebbAbelianNonAbelian}. The following proposition is due to Webb \cite[Equation (35) and discussion below]{WebbAbelianNonAbelian}
	\begin{pro}\label{connected_components}
		For every orbit $\mathfrak{o}$ of the Weyl group $W$ on the set $\lbrace \tilde{\delta} \in \chi(T)^\vee \, : \, \tilde{\delta} \mapsto \delta \rbrace$ pick a representative $\tilde{\delta}_{\mathfrak{o}}$. Then the map
		\begin{align*}
			k: \coprod_{\mathfrak{o}} \mathcal{F}_{0,\tilde{\delta}_{\mathfrak{o}}} \rightarrow F_{0,\delta},
		\end{align*}
		obtained by collecting all the morphisms $k$ defined above in a disjoint union, is a decomposition of $F_{0,\delta}$ into its connected components.
	\end{pro}	
	\subsection{Perfect obstruction theories}
	Since $\mathcal{F}_{0, \tilde{\delta}}$ is a connected component of $F_{0,\delta}$, the perfect obstruction theory on $Q(V/\!/G, \delta)$ induces one on $\mathcal{F}_{0,\tilde{\delta}}$ as described in \cite{Graber}
	\begin{align*}
		\mathbb{E}_{\vert \mathcal{F}_{0, \tilde{\delta}}}^{\mathbb{C}^\ast} \rightarrow \left(\mathbb{L}_{Q(V/\!/G, \delta)}\right)_{\vert \mathcal{F}_{0, \tilde{\delta}}}^{\mathbb{C}^\ast} \rightarrow \mathbb{L}_{\mathcal{F}_{0, \tilde{\delta}}},
	\end{align*}
	where $\mathbb{E} \rightarrow \mathbb{L}_{Q(V/\!/G, \delta)}$ is the obstruction theory on the quasimap moduli space. The virtual normal bundle is the moving part $\mathbb{E}\vert^\text{mov}_{\mathcal{F}_{0, \tilde{\delta}}}$. 
	
	Given $\tilde{\delta} \in \chi(T)^\vee$ consider the subset
	\begin{align*}
		\Delta(\tilde{\delta}) := \lbrace \alpha\in \Delta \, \vert \,  \langle\tilde{\delta}, \alpha \rangle <0 \rbrace
	\end{align*}
	of the set $\Delta$ of roots of $G$. In \cite[Corollary 5.2.3]{WebbAbelianNonAbelian} Webb proves the following
	\begin{lem}\label{psi_lemma}
		Let $\tilde{\delta}\mapsto \delta$. Then
		\begin{align*}
			\psi^\ast [\mathcal{F}_{0,\tilde{\delta}}]^\text{vir}_\text{loc} = h^\ast \left(\overline{K(\tilde{\delta})} \cap [F_{0,\tilde{\delta}}]^\text{vir}_\text{loc} \right)
		\end{align*}
		where
		\begin{align*}
			K(\tilde{\delta}):= \text{ev}^\ast\left(C(\tilde{\delta}) \prod_{\alpha \in \Delta(\tilde{\delta})}c_1(\mathcal{L}_\alpha)\right).
		\end{align*}
	\end{lem} 
	\begin{rem}
		As discussed in Remark \ref{action_inversion}, the original statement of Webb is about fixed loci of type $(\tilde{\delta}, 0)$. The statement for $(0,\tilde{\delta})$ is the same but with $z$ replaced by $-z$, which is why the class $K(\tilde{\delta})$ of Webb's result is replaced by $\overline{K(\tilde{\delta})}$ here.
	\end{rem}
	\subsection{Relating the Chern-Weil classes}
	One could wonder how the Chern-Weil homomorphism described above behaves with respect to abelianisation. Recall that we have the two homomorphisms
	\begin{align*}
		CW^{\tilde{\delta}} : A^\ast_T(\text{pt}) \rightarrow A^\ast_{\mathbb{C}^\ast}(\text{Q}(V/\!/T, \tilde{\delta}))\\
		CW^\delta : A^\ast_G(\text{pt}) \rightarrow A^\ast_{\mathbb{C}^\ast}(\text{Q}(V/\!/G, \delta))
	\end{align*}
	\begin{lem}\label{psiCW}
		Given $\delta \in \chi(G)^\vee$ and $\tilde{\delta} \mapsto \delta$ we have
		\begin{align*}
			\psi^\ast \text{CW}^\delta_{\vert \mathcal{F}_{0,\tilde{\delta}}} = h^\ast \text{CW}^{\tilde{\delta}}_{\vert {F}_{0,\tilde{\delta}}}
		\end{align*}
	\end{lem}
	\begin{proof}
		The morphism $\psi$ is defined via the universal property of $Q(V/\!/G, \delta)$ by using the quasimap $\mathcal{P}_{\tilde{\delta}} \times_T G \rightarrow F^0_{0,\tilde{\delta}} \times \mathbb{P}^1$. Restricting to the point $\infty \in \mathbb{P}^1$, this gives the commutative diagram
		\[\begin{tikzcd}
			\mathcal{P}_{\tilde{\delta} \vert \infty} \arrow[r,  "\Psi"] \arrow[d] & \mathcal{P}_{\delta \vert \infty} \arrow[d] \\
			F^0_{0,\tilde{\delta}} \arrow[r, swap, "\psi"] & \mathcal{F}_{0,\tilde{\delta}} 
		\end{tikzcd}\]
		between the universal quasimaps for $T$ and $G$ restricted at the relevant fixed loci of the quasimap spaces. This shows that $\mathcal{P}_{\tilde{\delta}\vert \infty}$ is just a reduction of $\psi^\ast \mathcal{P}_{\delta\vert \infty}$ to the structure group $T$, hence these two bundles have the same Chern-Weil homomorphisms.
	\end{proof}
	\subsection{The computation}
	Fix a dual character $\delta \in \chi(G)^\vee$ and consider the inverse image through the restriction map $\chi(T)^\vee \rightarrow \chi(G)^\vee$. The Weyl group acts on this set, and for each Weyl orbit $\mathfrak{o}$ in $\lbrace \tilde{\delta} \in \chi(T)^\vee \, \vert \, \tilde{\delta} \mapsto \delta \rbrace$ pick a representative $\tilde{\delta}_{\mathfrak{o}}$.
	Notice that, by merging together (\ref{openDiagram}) and (\ref{pfDiagram}), we obtain the following commutative diagram:
	\begin{equation}\label{fullDiagram}
		\begin{tikzcd}
			F_{0,\tilde{\delta}_{\mathfrak{o}}}\arrow[d, "\text{ev}"]\arrow[dr, phantom, "\square"] & F^0_{0,\tilde{\delta}_{\mathfrak{o}}} \arrow[l, swap, hook, "h"] \arrow[r, "\psi"] \arrow[d, "\text{ev}"]\arrow[rd, phantom, "\square"] & \mathcal{F}_{0,\tilde{\delta}_{\mathfrak{o}}} \arrow[d, "i"] \arrow[r, hook, "k"] & {F}_{0,\delta} \arrow[d, swap, "\text{ev}"]\\
			V/\!/T & V(G)^\text{ss}/\!/T \arrow[l, swap, hook, "j"]\arrow[r, "p"] \arrow[rr, bend right=30, "g"] & V(G)^\text{ss}/P_{\tilde{\delta}_{\mathfrak{o}}} \arrow[r, "f"] & V/\!/G .
		\end{tikzcd}
	\end{equation}
	This diagram can be useful to keep track of the spaces involved in the following computation.
	We want to compute the pullback via $g : V(G)^\text{ss}/\!/T \rightarrow V/\!/G$ of
	\begin{align}\label{compEquation1}
		&\text{ev}_\ast \left(\text{CW}^\delta(P) \cap [F_{0, \delta}]^\text{vir}_\text{loc}\right) = \sum_{\mathfrak{o}} f_\ast i_\ast \left(\text{CW}^\delta(P) \cap [\mathcal{F}_{0, \tilde{\delta}_{\mathfrak{o}}}]^\text{vir}_\text{loc}\right)
	\end{align}
	where the sum is over the Weyl orbits described above, and the equality follows from Lemma \ref{connected_components} and commutativity of the right square in diagram (\ref{fullDiagram}).
	We are left to prove the following
	\begin{lem}
		In the notation above
		\begin{align*}
			g^\ast f_{\ast} i_{\ast} \left(\text{CW}^{\delta}(P)\cap [\mathcal{F}_{0, \tilde{\delta}_{\mathfrak{o}}}]^\text{vir}_\text{loc}\right) = j^\ast\sum_{\tilde{\delta} \in \mathfrak{o}} &\overline{ C(\tilde{\delta})} \cap \text{ev}_\ast \left(\text{CW}^{\tilde{\delta}}(P)\cap [F_{0, \tilde{\delta}}]^\text{vir}_\text{loc}\right).
		\end{align*}
	\end{lem}
	\begin{proof}
		By Brion's lemma \cite[Proposition 2.1]{BrionFlagBundles} in the form of \cite[Lemma 5.3.1]{WebbAbelianNonAbelian} we can write the left-hand side as
		\begin{align*}
			\sum_{w \in W/\text{stab}(\tilde{\delta}_{\mathfrak{o}})}  (w^{-1})^\ast \frac{p^\ast i_{\ast}\left(\text{CW}^{\delta}(P)\cap [\mathcal{F}_{0, \tilde{\delta}_{\mathfrak{o}}}]^\text{vir}_\text{loc}\right)}{\prod_{\alpha \in \Delta(\tilde{\delta}_{\mathfrak{o}})}c_1(\mathcal{L}_\alpha)}, 
		\end{align*}
		where $\Delta(\tilde{\delta}_{\mathfrak{o}}) = \lbrace \alpha \in \Delta \, \vert \, \langle \tilde{\delta}_{\mathfrak{o}}, \alpha \rangle <0 \rbrace$ and $w: V(G)^\text{ss}/\!/T \xrightarrow{\sim} V(G)^\text{ss}/\!/T$ is given by the action of the Weyl group on this space\footnote{Notice that the stabiliser of $\tilde{\delta}_{\mathfrak{o}}$ doesn't act trivially on $V(G)^{\text{ss}}/\!/T$. The pullback through $w^{-1}$ written above still makes sense since the class we are pulling back is invariant for the action of $\text{stab}(\tilde{\delta}_{\mathfrak{o}})$ by construction.}. Since the middle square in (\ref{fullDiagram}) is Cartesian we can write this contribution as
		\begin{align*}
			\sum_{w \in W/\text{stab}(\tilde{\delta}_{\mathfrak{o}})}  (w^{-1})^\ast \frac{\text{ev}_\ast \psi^\ast \left(\text{CW}^{\delta}(P)\cap [\mathcal{F}_{0, \tilde{\delta}_{\mathfrak{o}}}]^\text{vir}_\text{loc}\right)}{\prod_{\alpha \in \Delta(\tilde{\delta}_{\mathfrak{o}})}c_1(\mathcal{L}_\alpha)}.
		\end{align*}
		By Lemma \ref{psi_lemma} and Lemma \ref{psiCW} we know how to pullback by $\psi$, thus getting
		\begin{align*}
			\sum_{w \in W/\text{stab}(\tilde{\delta}_{\mathfrak{o}})}  &(w^{-1})^\ast \frac{\text{ev}_\ast h^\ast \left(\overline{K(\tilde{\delta}_{\mathfrak{o}})}\text{CW}^{\tilde{\delta}_{\mathfrak{o}}}(P)\cap [F_{0, \tilde{\delta}_{\mathfrak{o}}}]^\text{vir}_\text{loc}\right)}{\prod_{\alpha \in \Delta(\tilde{\delta}_{\mathfrak{o}})}c_1(\mathcal{L}_\alpha)}.
		\end{align*}
		Finally, the Weyl group action commutes with everything as discussed in \cite[Section 5.3]{WebbAbelianNonAbelian}, hence this class coincides with
		\begin{align*}
			\sum_{w \in W/\text{stab}(\tilde{\delta}_{\mathfrak{o}})}  &\frac{\text{ev}_\ast h^\ast \left(\overline{K(w^{-1}\cdot\tilde{\delta}_{\mathfrak{o}})}\text{CW}^{w^{-1}\cdot\tilde{\delta}_{\mathfrak{o}}}(P)\cap [F_{0, w^{-1}\cdot\tilde{\delta}_{\mathfrak{o}}}]^\text{vir}_\text{loc}\right)}{\prod_{\alpha \in \Delta(w^{-1}\cdot\tilde{\delta}_{\mathfrak{o}})}c_1(\mathcal{L}_\alpha)}\\
			= j^\ast\sum_{\tilde{\delta} \in \mathfrak{o}} &\frac{\text{ev}_\ast \left(\overline{K(\tilde{\delta})}\text{CW}^{\tilde{\delta}}(P)\cap [F_{0, \tilde{\delta}}]^\text{vir}_\text{loc}\right)}{\prod_{\alpha \in \Delta(\tilde{\delta})}c_1(\mathcal{L}_\alpha)}\\
			= j^\ast\sum_{\tilde{\delta} \in \mathfrak{o}} &\overline{ C(\tilde{\delta})} \cap \text{ev}_\ast \left(\text{CW}^{\tilde{\delta}}(P)\cap [F_{0, \tilde{\delta}}]^\text{vir}_\text{loc}\right),
		\end{align*}
		where the first equality holds by identifying $W/\text{stab}(\tilde{\delta}_{\mathfrak{o}}) = \mathfrak{o}$ via the representative $\tilde{\delta}_\mathfrak{o}$, and the last one is just an immediate application of the projection formula.
	\end{proof}
	This concludes the proof of Lemma \ref{generalisation_webb}, and thus of Proposition \ref{AbelianisationClasses}.
	\section{More on the proof of Theorem \ref{theoremQuestion2}}\label{sectionSzenesVergne}
	Hoping to help the reader, we recall the main result \cite[Theorem 4.1]{SzenesVergne} by Szenes-Vergne. We are interested in the following form of their theorem:
	\begin{theorem}\label{SzenesVergneVers1}
		Consider a finite dimensional representation $V$ of an algebraic torus $T$ together with a GIT stability $\xi \in \chi(T)$ so that the action on the semistable locus is free and $V/\!/T$ is proper and semi-positive. Let $P \in A^\ast_T(\text{pt})$ and consider the formal sums
		\begin{align}\label{functionBSide}
			\sum_{w \in p^{-1}(q)} \frac{P(w)}{D_{\mathfrak{A}}(w)\prod_{\rho \in \mathfrak{A}}\rho(w)},
		\end{align}
		and
		\begin{align*}
			\langle P \rangle^T(q) := \sum_{\lambda \in \chi(T)^\vee} q^\lambda \int_{[Q(V/\!/T,\lambda)]^\text{vir}} \text{CW}^\lambda(P).
		\end{align*}
		Then the following hold true:
		\begin{enumerate}
			\item The sum (\ref{functionBSide}) is well defined for a generic $q$, and it extends to a rational function $\langle P\rangle_{\mathfrak{B}}$ on $\widecheck{T}$ which is well defined on the domain of convergence of $\langle P \rangle^T$.
			\item $\langle P \rangle^T(q)$ converges on a nonempty open (in the analytic topology) subset of $\widecheck{T}$, and on its domain of convergence it coincides with the function $\langle P \rangle_{\mathfrak{B}}(q)$.
		\end{enumerate}
	\end{theorem}
	\subsection{Szenes-Vergne's original theorem}
	Szenes and Vergne prove a different statement which implies the one above. Let's discuss their original version. 
	For every weight $\rho \in \mathfrak{A}$, consider the corresponding 1-dimensional representation $\mathbb{C}_\rho \simeq \mathbb{C}$. From these we can build the representation $\bigoplus_{\rho \in \mathfrak{A}}\mathbb{C}_\rho$ (denoted with $\mathfrak{g}$ in \cite{SzenesVergne}), which is abstractly isomorphic to $V$ (remember that weights appear repeated with their multiplicity inside $\mathfrak{A}$). This representation contains a torus, namely $\bigoplus_{\rho \in \mathfrak{A}}\mathbb{C}^\ast_\rho$, where we use the notation $\mathbb{C}_\rho^\ast := \mathbb{C}_\rho\setminus \lbrace0\rbrace \simeq \mathbb{C}^\ast$.
	Given $z \in \bigoplus_{\rho\in \mathfrak{A}} \mathbb{C}^\ast_\rho$ we can consider the following formal sum
	\begin{align*}
		\langle P \rangle_{\mathfrak{A}, \mathfrak{c}}(z):= \sum_{\lambda \in \chi(T)^\vee} z^\lambda \int_{[Q(V/\!/T, \lambda)]^\text{vir}} \text{CW}^\lambda(P)
	\end{align*}
	where $z^\lambda := \prod_{\rho \in \mathfrak{A}} z_\rho^{\langle \lambda, \rho \rangle}$. Consider the following submersive morphism\footnote{This is just the morphism $\mathbb{C}^n\rightarrow \mathbb{C}^r$ sending $x \mapsto \sum_\rho x_\rho \rho$ in logarithmic coordinates.}
	\begin{align}\label{uMorphism}
		u : \bigoplus_{\rho\in \mathfrak{A}} \mathbb{C}^\ast_\rho \twoheadrightarrow \widecheck{T} \quad : \quad u(z) := \sum_{\rho \in \mathfrak{A}} \frac{\log(z_\rho)}{2\pi i} [\rho].
	\end{align}
	The result of Szenes and Vergne is the following:
	\begin{theorem}[Szenes-Vergne]\label{SzenesVergneOriginal}
		\sloppy 
		Let the hypotheses of Theorem \ref{SzenesVergneVers1} hold true. The sum (\ref{functionBSide}), evaluated at $u(z)$, defines a rational function denoted with $\langle P \rangle_\mathfrak{B}(u(z))$, well defined on the domain of convergence of $\langle P \rangle_{\mathfrak{A}, \mathfrak{c}}$. Moreover, the sum $\langle P \rangle_{\mathfrak{A}, \mathfrak{c}}$ converges on a nonempty open subset of $\bigoplus_{\rho\in \mathfrak{A}} \mathbb{C}^\ast_\rho$ (in the analytic topology) and on its domain of convergence it coincides with the function $\langle P\rangle_{\mathfrak{B}}(u(z))$.
	\end{theorem}
	\begin{rem}
		In \cite{SzenesVergne}, the generating functions $\langle P \rangle$ differ from ours in that each coefficient of their series contains an additional cohomology class under the integral sign, called \textit{Morrison-Plesser class} of the quasimap space. Szenes and Vergne consider this additional class since they want to compute quasimap invariants of Calabi-Yau subvarieties cut as complete intersections in $V/\!/T$ using a quantum Lefschetz type of theorem. Here we only consider the case of $V/\!/T$, so the Morrison-Plesser class doesn't appear in our power series.
	\end{rem}
	How do we recover Theorem \ref{SzenesVergneVers1} from the original Theorem \ref{SzenesVergneOriginal}? First of all notice that the morphism $u$ satisfies the equality $z^\lambda = u(z)^\lambda$ for all $\lambda\in \chi(T)^\vee$, hence
	\begin{align}\label{equalityOfSVSeries}
		\langle P \rangle^T (u(z)) = \langle P \rangle_{\mathfrak{A}, \mathfrak{c}}(z)
	\end{align}
	holds true for every $z \in \bigoplus_{\rho \in \mathfrak{A}}\mathbb{C}_\rho^\ast$.	If $\langle P \rangle_{\mathfrak{A}, \mathfrak{c}}$ converges on a nonempty open subset of $\bigoplus_{\rho\in \mathfrak{A}} \mathbb{C}^\ast_\rho$, then by (\ref{equalityOfSVSeries}) the sum $\langle P \rangle^T$ converges in the image of that open subset through the map $u$, which is still a nonempty open subset of $\widecheck{T}$ being $u$ a submersion. Assume that $\langle P \rangle^T$ converges at some $q \in \widecheck{T}$. Since $u$ is surjective, we have that $q= u(z)$ for some $z \in \bigoplus_{\rho \in \mathfrak{A}}\mathbb{C}_\rho^\ast$ and hence $\langle P \rangle_{\mathfrak{A}, \mathfrak{c}}$ converges at $z$ by (\ref{equalityOfSVSeries}). If Theorem \ref{SzenesVergneOriginal} holds true, then we have
	\begin{align*}
		\langle P \rangle^T (q) = \langle P \rangle_{\mathfrak{A}, \mathfrak{c}}(z) = \langle P \rangle_\mathfrak{B} (u(z)) = \langle P \rangle_\mathfrak{B} (q).
	\end{align*}
	It remains to prove that $\langle P \rangle_\mathfrak{B}$ defines a rational function on $\widecheck{T}$, but this follows from the following two facts:
	\begin{enumerate}
		\item The function $\langle P \rangle_\mathfrak{B} \circ u$ is rational by Theorem \ref{SzenesVergneOriginal}.
		\item The morphism $u$ is smooth and surjective, hence we can invoke the descent theorem for morphisms of schemes under fpqc morphisms \cite[Theorem 4.33]{vistoli2004notes}.
	\end{enumerate}
	This completes the proof of Theorem \ref{SzenesVergneVers1}.
	\printbibliography
\end{document}